\documentclass[a4paper,12pt,fleqn]{article}
\usepackage[latin1]{inputenc}  \usepackage[T1]{fontenc}
\usepackage{lscape}
\usepackage{tabularx}     
\usepackage{pstricks}    
\usepackage{rotating}
\usepackage{amsfonts,amsmath,amsthm,amssymb,dsfont}
\usepackage{marginnote}
\usepackage{rotate}
\usepackage[arrow,matrix,curve]{xy} 	
\title{Non-congruence of homology Veech groups in genus two}
\author{Christian Weiß\thanks{The author is partially supported by the ERC-StG 257137.}}
\date{\today}

\newtheorem{thm}{Theorem}[section]
\newtheorem*{thm3}{Theorem~\ref{thm:main_thm}}
\newtheorem*{thm4}{Theorem~\ref{thm:weierstrass}}

\newtheorem{rem}[thm]{Remark}
\newtheorem{lem}[thm]{Lemma}
\newtheorem{prop}[thm]{Proposition}
\newtheorem{cor}[thm]{Corollary}

\newtheorem{exa}[thm]{Example}

\newcommand{\SL}{{\rm{SL}}}

\newcommand{\Id}{{\rm{Id}}}
\newcommand{\PSL}{{\rm{PSL}}}
\newcommand{\GL}{{\rm{GL}}}

\newcommand{\Jac}{{\rm{Jac}}}

\newcommand{\diag}{{\rm{diag}}}
\newcommand{\tr}{{\rm{tr}}}

\newcommand{\M}{{\rm{M}}}
\newcommand{\N}{{\mathcal{N}}}
\newcommand{\Aff}{{\rm{Aff}}}
\newcommand{\HH}{{\mathbb{H}}}
\newcommand{\CC}{{\mathbb{C}}}
\newcommand{\RR}{{\mathbb{R}}}
\newcommand{\ZZ}{{\mathbb{Z}}}
\newcommand{\NN}{{\mathbb{N}}}
\newcommand{\QQ}{{\mathbb{Q}}}

\newcommand{\Mg}{{\mathcal{M}_g}}
\newcommand{\Mtwo}{{\mathcal{M}_2}}
\newcommand{\OD}{\mathcal{O}_D}

\newenvironment{bsmallmatrix}
{\left[\begin{smallmatrix}}
	{\end{smallmatrix}\right]}

\newenvironment{psmallmatrix}
{\left(\begin{smallmatrix}}
	{\end{smallmatrix}\right)}

\makeindex

\begin{document} 

\maketitle

\begin{abstract}
We study the action of the Veech group of square-tiled surfaces of genus two on homology. 
This action defines the homology Veech group which is a subgroup of $\SL_2(\OD)$ where $\OD$ is a quadratic order
of square discriminant. Extending a result of Weitze-Schmithüsen we show that also the homology Veech group is 
a totally non-congruence subgroup with exceptions stemming only from the prime ideals lying above~$2$. While
Weitze-Schmithüsen's result for Veech groups is asymmetric with respect to the spin structure our use of the homology Veech group yields 
a completely symmetric picture.
\end{abstract}

\tableofcontents

\section{Introduction} \label{cha:Introduction}
Veech groups of square-tiled surfaces are an interesting class of subgroups of $\SL_2(\ZZ)$. They consist of the derivatives of affine diffeomorphisms of a finite number of glued copies of the unit square (see e.g. \cite{WS05}). Their study goes back to the fundamental paper \cite{Vee89}, where Veech groups have been introduced first. Ellenberg and McReynolds showed in \cite{ER12} that with some minor restrictions all subgroups of $\SL_2(\ZZ)$ appear as a Veech group if the genus
of the square-tiled surface is allowed to be large. However, for genus two square-tiled surfaces Weitze-Schmithüsen
proved in \cite{WS12} that these Veech groups are very far away from being congruence subgroups. In fact, the Veech group
of a genus two square-tiled surface has two representations in $\SL_2(\ZZ)$ given by its action on homology. These two
representations yield a pair of matrices $(A_1,A_2) \in \SL_2(\OD)$ where $\OD$ is a quadratic order of square discriminant
(see Section~\ref{cha:special_linear_group}). The image of this representation is called homology Veech group. In this paper, we generalize 
Weitze-Schmithüsen's result by showing that also the homology Veech group is very far away from being a congruence 
subgroup.\\[12pt]
For our result Weitze-Schmithüsen's approach of using different prototypes in the sense of \cite{Bai07} or \cite{McM05} at
the same time seems to be of limited use. 
Moreover her idea of dealing with the Wohlfahrt level may not be easily carried over to subgroups of $\SL_2(\OD)$ since
it is not evident how to generalize the Wohlfahrt level to this case. Indeed our additional ingredients
are therefore that we explicitly find coset representatives like in \cite{Wei12} and use Nori's Theorem, which states that 
the number of elements of  a subgroup of $\SL_2(\mathbb{F}_p)$ generated by parabolic elements is limited to only three 
possibilities (Corollary~\ref{cor:Nori}).\\[12pt] 
While Weitze-Schmithüsen takes into account principal congruence subgroups we look at Hecke congruence subgroups instead. Nori's Theorem \cite[Theorem~5.1]{Nor87} implies that the assertion for principal congruence subgroups follows from the case of Hecke congruence subgroups at least for all prime ideals. 
So on the level of Veech groups the two statements are almost equivalent and considering Hecke congruence subgroups is not a great restriction. Furthermore Weitze-Schmithüsen's theorem is asymmetric with respect to the spin structure of the square-tiled surfaces but our approach of using the homology Veech group symmetrizes the result.\\[12pt]
Let us now make more precise what we mean by ``being very far away from being a congruence subgroup``: 
let $\mathcal{O}$ be either $\ZZ$ or more generally a quadratic order $\OD$. Recall that an ideal is called \textbf{regular} if it contains a non-zero divisor. A finite index subgroup of $\SL_2(\mathcal{O})$ is called a \textbf{congruence subgroup}
if, for some regular ideal $\mathfrak{a}$ it contains a principal congruence subgroup 	$$\Gamma(\mathfrak{a}) := \left\{ \begin{pmatrix} a & b \\ c & d \end{pmatrix} \in \SL_2(\OD): \begin{pmatrix} a & b \\ c & d \end{pmatrix} 
\equiv \begin{pmatrix} 1 & 0 \\ 0 & 1 \end{pmatrix} \mod \mathfrak{a} \right\}.$$
(see Section~\ref{cha:special_linear_group} for details). Observe that $\Gamma$ is a congruence subgroup of $\SL_2(\mathcal{O})$ if and only if there exists a regular ideal $\mathfrak{a} \subset \mathcal{O}$ such 
that the \textbf{level index} $[\SL_2(\mathcal{O}/\mathfrak{a}) : \rho_{\mathfrak{a}}(\Gamma)]$ equals the index 
$[\SL_2(\mathcal{O}):\Gamma]$ where $\rho_{\mathfrak{a}}: \SL_2(\mathcal{O}) \to \SL_2(\mathcal{O}/\mathfrak{a})$ is the natural
projection. The group $\Gamma$ is thus called a \textbf{non-congruence subgroup of level $\mathfrak{a}$} if the two indices differ 
and $\Gamma$ is called a \textbf{totally non-congruence subgroup of level $\mathfrak{a}$} if 
$[\SL_2(\mathcal{O}/\mathfrak{a}) : \rho_{\mathfrak{a}}(\Gamma)]=1$. Being a totally non-congruence subgroup of level
$\mathfrak{a}$ is equivalent to the index $[\Gamma:\Gamma \cap \Gamma(\mathfrak{a})]$ being equal to the index 
$[\SL_2(\mathcal{O}):\Gamma(\mathfrak{a})].$\\[12pt]
Let $\mathcal{M}_g$ denote 	the moduli space of Riemann surfaces of genus $g$. A \textbf{translation surface} $(X,\omega) \in \Omega \mathcal{M}_g$ is a pair consisting of some $X \in \mathcal{M}_g$ equipped with a non-zero holomoprhic $1$-form $\omega \in \Omega(X)$ on $X$. The vector-bundle $\Omega \mathcal{M}_g$ is stratified by decomposing it into equivalence classes of pairs $(X,\omega) \in \Omega \mathcal{M}_g (k_1,\ldots,k_n)$ with $\sum k_i = 2g-2$, where $\omega$ has exactly $n$ zeroes with multiplicities $k_1,\ldots,k_n$. The group $\SL_2(\RR)$ naturally acts on $\Omega\mathcal{M}_g (k_1,\ldots,k_n)$: for $A = \begin{psmallmatrix} a & b \\ c & d \end{psmallmatrix} \in \SL_2(\RR)$ consider the holomorphic one-form
$$\omega' = \begin{pmatrix} 1 & i \end{pmatrix} \begin{pmatrix} a & b \\ c & d \end{pmatrix} \begin{pmatrix} \text{Re}(\omega) \\ \text{Im}(\omega) \end{pmatrix}$$
on $X$. There is a unique complex structure with respect to which $\omega'$ is holomoprhic. Its charts yield a new Riemann surface $X'$ and $A \cdot (X, \omega) := (X',\omega').$ Square-tiled surfaces build an important class of examples of translation surfaces: the surface $X$ is obtained as a covering of the unit torus ramified over one point only and $\omega$ is the pullback of the holomorphic $1$-form of the torus. \\[12pt]
Weitze-Schmith\"usen proved in \cite{WS12} the following theorem on Veech groups of square-tiled surfaces in $\Omega\Mtwo(2)$. Recall that the orbits of square-tiled surfaces in $\Omega\Mtwo(2)$ under the action of $\SL_2(\RR)$ can be classified into two sets $\mathcal{A}_d$ and $\mathcal{B}_d$ depending on their spin structure (even or odd) which is detemined by the number of integral Weierstra\ss \ points, see Theorem~\ref{cor:classification}.
 
\begin{thm} \textbf{(Weitze-Schmithüsen, \cite[Theorem~3]{WS12})}
  Let $L_d$ be a square-tiled surface in $\Omega\Mtwo(2)$ with $d$ squares and let $\SL(L_d)$ be its Veech group. We distinguish the two different
  cases that $L_d$ is in the orbit $\mathcal{A}_d$ and $\mathcal{B}_d$ in the classification of square-tiled surfaces in $\Omega \Mtwo(2)$.
  \begin{itemize}
   \item[(1A)] If $d$ is even or $d$ is odd and $L_d$ is in $\mathcal{A}_d$, or $d=3$, then we have
   $[\SL(L_d):\SL(L_d) \cap \Gamma(n)] = [\SL_2(\ZZ):\Gamma(n)]$ for all odd $n \in \NN$.
   \item[(1B)] If $d$ is even or $d$ is odd and $L_d$ is in $\mathcal{A}_d$, or $d=3$, then we have
   $[\SL(L_d):\SL(L_d) \cap \Gamma(n)] = \tfrac{1}{3}[\SL_2(\ZZ):\Gamma(n)]$ for all even $n \in \NN$.
   \item[(2)] If $d$ is odd, $d \geq 5$ and $L_d$ is in $\mathcal{B}_d$ then $\SL(L_d)$ is a totally non-congruence subgroup,
   i.e.  $[\SL(L_d):\SL(L_d) \cap \Gamma(n)] = [\SL_2(\ZZ):\Gamma(n)]$ for all $n \in \NN$.
  \end{itemize}
\end{thm}
Note that the theorem is asymmetric with respect to the spin structure: while there is need for distinction between odd and even $n$ in $\mathcal{A}_d$, this is not necessary in $\mathcal{B}_d$. If one introduces the Hecke congruence subgroups $\Gamma_0(n)$ where only the lower left entry of the matrices has to be equal to $0 \mod n$, the theorem implies:

\begin{cor} \label{cor:schmitt}
  Let $L_d$ be a square-tiled surface in $\Omega\Mtwo(2)$ with $d$ squares and let $\SL(L_d)$ be its Veech group. 
  \begin{itemize}
   \item[(1A)] If $d$ is even, or $d$ is odd and $L_d$ is in $\mathcal{A}_d$, or $d=3$, then we have
   $[\SL(L_d):\SL(L_d) \cap \Gamma_0(n)] = [\SL_2(\ZZ):\Gamma_0(n)]$ for all odd $n \in \NN$.
   \item[(1B)] If $d$ is even, or $d$ is odd and $L_d$ is in $\mathcal{A}_d$, or $d=3$, then we have
   $[\SL(L_d):\SL(L_d) \cap \Gamma_0(n)] = c [\SL_2(\ZZ):\Gamma_0(n)]$ for all even $n \in \NN$ with $c \in \left\{ \tfrac{1}{3}, \tfrac{3}{3} \right\}$.
   \item[(2)] If $d$ is odd, $d \geq 5$ and $L_d$ is in $\mathcal{B}_d$ then $[\SL(L_d):\SL(L_d) \cap \Gamma_0(n)] = [\SL_2(\ZZ):\Gamma_0(n)]$ for all $n \in \NN$.
  \end{itemize}
\end{cor}
In fact, both values of $c$ in (1B) occur depending on the choice of the square-tiled surface and hence its embedding into $\SL_2(\ZZ)$ since $\Gamma_0(n)$ is not normal for $n >1$. Conversely, Nori's Theorem \cite[Theorem~5.1]{Nor87} yields that the assertion of the theorem follows from the corollary for all prime numbers $p \in \NN$.\\[12pt]
In this paper, we generalize Weitze-Schmithüsen's theorem in the following sense: a square-tiled surface $\pi_1: X \to E_1$ with $g(X)=2$ and  consisting of $d$ unit squares is called \textbf{minimal} if it does not factor via an isogeny. For a minimal square-tiled surface there exists a covering $\pi_2: X \to E_2$ of the same degree such that the induced morphism $\Jac(X) \to E_1 \times E_2$ is an isogeny (of degree $d^2$).
We call $E_2$ the \textbf{complementary elliptic curve} (see Section~\ref{Sec:TMC}, \cite{BL04} or \cite{Kan03} for details). This means that $H_1(X,\ZZ)$  contains $\Lambda:= H_1(E_1,\ZZ) \oplus H_1(E_2,\ZZ)$ as a sublattice of index $d^2$ and that the symplectic pairing on $H_1(X,\ZZ)$ respects this decomposition. The action of the Veech group on $X$ induces an action on homology and thus an action on $\Lambda$. 
For a given matrix $A$ in the Veech group the induced action on $\Lambda$ is given by a pair of matrices $\widetilde{A}:=(A_1,A_2)$, where $A_1$ acts on $H_1(E_1,\ZZ)$ and $A_2$ acts on $H_1(E_2,\ZZ)$. More precisely, it follows from \cite{Kan03} that $\widetilde{A}$ is given by $(A_1=A,RA_2R^{-1})$ and that $\widetilde{A}$ is a matrix in $\SL_2(\mathcal{O}_{d^2})$ where $\mathcal{O}_{d^2}$ is the quadratic order
of discriminant $d^2$. We call the corresponding subgroup of $\SL_2(\mathcal{O}_{d^2})$ the \textbf{homology Veech group}. 
In particular, the homology Veech group is a subgroup of a Lie group of rank~$2$ while the Veech group is only a subgroup 
of a Lie group of rank $1$. However, the homology Veech group is a subgroup of $\SL_2(\mathcal{O}_{d^2})$ of infinite index. 
Recall that the notion of being a totally non-congruence subgroup of level $\mathfrak{a}$ is yet equivalent to the property $[\Gamma:\Gamma \cap \Gamma(\mathfrak{a})]=[\SL_2(\mathcal{O}):\Gamma(\mathfrak{a})]$ which still makes sense when $\Gamma$ is of infinite index.\\[12pt]
Given those cases where Weitze-Schmithüsen's result implies that the Veech group is a totally non-congruence subgroup it is most interesting to analyze primes where this fails for the homology Veech group. 
The asymmetry with respect to the spin structure which appeared in the theorem of
Weitze-Schmithüsen then vanishes (Theorem~\ref{thm:main_thm}). The topological reason for this to happen is that the number of integral Weierstra\ss \ points (spin structure) is well-defined not only for the square-tiled covering map but also for the complementary covering as we will prove.
\begin{thm4} Let $\pi_1 : L_d \to E_1$ be a square-tiled surface in $\Omega \Mtwo(2)$ 
and let $E_2$ denote the complementary elliptic curve. If $\pi_1 : L_d \to  E_1$ has even spin, then $\pi_2 : L_d \to  E_2$ has odd spin. If $\pi_1 : L_d \to E_1$ has odd spin, then $\pi_2 : L_d \to E_2$ has even spin.
\end{thm4} 
This result allows us to preserve the main properties of Weitze-Schmithü-sen's theorem but to solve its strange asymmetry
with respect to the spin structure.

\begin{thm3}
  Let $L_d$ be a square-tiled surface in $\Omega\Mtwo(2)$ with $d$ squares and let $\SL^1(L_d)$ be its homology Veech group. We distinguish the two different cases that $L_d$ is in the orbit $\mathcal{A}_d$ and $\mathcal{B}_d$ in the classification of square-tiled surfaces in $\Omega \Mtwo(2)$. Moreover if $d$ is odd let $2=\mathfrak{p}_2\mathfrak{p}_2^\sigma$ 
  be the decomposition of~$2$ into prime ideals, where $\mathfrak{p}_2$ is the distinguished prime ideal that is a common divisor of~$2$
   and $(2,2-d)$ in $\mathcal{O}_{d^2}$. 
  \begin{itemize}
   \item[(1A)] If $d$ is odd and $L_d$ is in $\mathcal{A}_d$, or $d=3$, then
   $[\SL^1(L_d):\SL^1(L_d) \cap \Gamma_0(\mathfrak{a})] = [\SL_2(\mathcal{O}_{d^2}):\Gamma_0(\mathfrak{a})]$ for all ideals $\mathfrak{a} \subset \mathcal{O}_{d^2}$ with
   $\gcd(\mathfrak{p}_2,\mathfrak{a})=1$.
   \item[(1B)] If $d$ is odd and $L_d$ is in $\mathcal{A}_d$, or $d=3$, then $[\SL^1(L_d):\SL^1(L_d) \cap \Gamma_0(\mathfrak{a})] = \frac{2}{3}[\SL_2(\mathcal{O}_{d^2}):\Gamma_0(\mathfrak{a})]$
   for all ideals with $\mathfrak{p}_2|\mathfrak{a}$.
   \item[(2A)] If $d$ is odd and $L_d$ is in $\mathcal{B}_d$, then we have
   $[\SL^1(L_d):\SL^1(L_d) \cap \Gamma_0(\mathfrak{a})] = [\SL_2(\mathcal{O}_{d^2}):\Gamma_0(\mathfrak{a})]$ for all ideals $\mathfrak{a} \subset \mathcal{O}_{d^2}$ with
   $\gcd(\mathfrak{p}_2^\sigma,\mathfrak{a})=1$.
   \item[(2B)] If $d$ is odd and $L_d$ is in $\mathcal{B}_d$, then we have $[\SL^1(L_d):\SL^1(L_d) \cap \Gamma_0(\mathfrak{a})] = \frac{2}{3}[\SL_2(\mathcal{O}_{d^2}):\Gamma_0(\mathfrak{a})]$
   for all ideals with $\mathfrak{p}_2^\sigma|\mathfrak{a}$.
   \item[(3A)] If $d$ is even, then 
   $[\SL^1(L_d):\SL^1(L_d) \cap \Gamma_0(\mathfrak{a})] = [\SL_2(\mathcal{O}_{d^2}):\Gamma_0(\mathfrak{a})]$ for all ideals $\mathfrak{a} \subset \mathcal{O}_{d^2}$ with
   $2 \nmid \N(\mathfrak{a})$ (the norm of $\mathfrak{a}$).
   \item[(3B)] If $d$ is even, then $[\SL^1(L_d):\SL^1(L_d) \cap \Gamma_0(\mathfrak{a})] = \frac{2}{3}[\SL_2(\mathcal{O}_{d^2}):\Gamma_0(\mathfrak{a})]$
   for all ideals with $2|\N(\mathfrak{a})$,
   \end{itemize}
\end{thm3}

For primitive Teichmüller curves in $\Omega\Mtwo(2)$, i.e. those not stemming from square-tiled surfaces, and fundamental discriminants,
a theorem which exactly corresponds to Theorem~\ref{thm:main_thm} was  proven in \cite[Theorem~5.1]{Wei12}. Therefore, this paper almost completes the picture for $\Omega\Mtwo(2)$ and $\Gamma_0$. Nori's theorem mainly closes the gap to principal congruence subgroups. To get a complete picture one only has to perform a similar analysis for those Teichmüller curves whose discriminants are neither square nor fundamental.

\paragraph{Acknowledgement.} I am very grateful to Martin Möller for his constant support of my work on this paper and
for many very fruitful discussions. Moreover I would like to thank Andr\'{e} Kappes for showing me how to practically calculate
elements of the homology Veech group and my former office mate Quentin Gendron for always being willing to discuss my mathematical problems. Finally, I thank the two referees for many useful comments.

\section{The Special Linear Group over square quadratic orders} \label{cha:special_linear_group}

In this section quadratic orders $\OD$ are introduced. We will almost exclusively deal here with the case of square discriminants. 
In particular, $\textrm{Spec} \ \OD$ is calculated (Theorem~\ref{prop:Spec}) and the ramification of prime numbers $p \in \NN$ over $\OD$ is discussed. Due to its similarity to the case of fundamental discriminants most of the proofs are postponed to the appendix. Afterwards, we will focus on the special linear group, define congruence subgroups and calculate some relevant indexes. 

\paragraph{Quadratic orders.} Let $K$ be a quadratic field or $\QQ \oplus \QQ$. A \textbf{quadratic order} is a subring $\mathcal{O}$ of $K$ that is also a finitely generated $\ZZ$-module such that $1 \in \mathcal{O}$ and $\mathcal{O} \otimes \QQ = K$. Any quadratic order is isomorphic to one of the form
$$\OD = \ZZ[T]/(T^2+bT+c),$$
where $b,c \in \ZZ$ and $b^2-4c=D$ and the isomorphism class only depends on the \textbf{discriminant} $D$ (see e.g. \cite[Chapter~2.2]{Bai07}). 
For every $D \equiv 0,1 \mod 4$ there hence exists a unique quadratic order. In this paper, we will only be concerned about 
the case where $D=d^2$ is a square. To the best of our knowledge, there does not seem to be any good reference, where these
special quadratic orders are treated in detail. Therefore we collect some elementary facts and explain similarities 
and differences to the non-square case. For further background the reader is referred to the appendix
\paragraph{Square discriminants.} Let us describe the structure of the quadratic order. As $D=d^2$ is a square we have that
$$\OD = \left\{ (x,y) \in \ZZ \times \ZZ \ | \ x \equiv y \mod d \right\}$$
since every subring of $\ZZ \oplus \ZZ$ is of this form.\footnote{This is true because every subring of $\ZZ$ is of the form $n\ZZ$.}
The quadratic order is hence a subring of $K = \QQ \oplus \QQ$, where addition and multiplication are defined componentwise. The algebra $K$ may be interpreted
as a substitute for the quadratic number field $\QQ(\sqrt{D})$ where $\OD$ is contained in for non-square $D$ (this
is also the reason for our notation). Furthermore we can regard $\QQ \oplus \QQ$ as an extension of $\QQ$ by the diagonal map
$\QQ \to \QQ \oplus \QQ$. This makes perfectly sense since $\ZZ$ is embedded into $\OD$ by this construction. The Galois
automorphism of $\QQ \oplus \QQ$ is given by
$$(x,y) \mapsto (x,y)^\sigma := (y,x).$$ 
\paragraph{Norm and trace.} The Galois automorphism is used to define norm and trace on $\OD$ and $\QQ \oplus \QQ$ respectively:
$$\N((x,y)):=(x,y)(x,y)^\sigma,$$
$$\tr((x,y)):=(x,y)+(x,y)^\sigma.$$
We call $\mathds{1}:=(1,1)$ and $w:=(0,d)$ the \textbf{standard  basis} of $\OD$ and they indeed generate $\OD$ as a $\ZZ-$module. Hence, $\OD$ is a Noetherian ring.

\paragraph{Ideals and modules.} For  $z\in \OD$ we denote the \textbf{principal ideal} generated by $z$ by $$(z):=z\OD:=\left\{za| a \in \OD\right\}.$$
Now let $(0) \neq \mathfrak{a} \subset \OD$ be an arbitrary ideal in $\OD$. Its norm $\N_D(\mathfrak{a})$ is 
defined as the number of the elements in $\OD/\mathfrak{a}$ if this quotient is finite. If the quotient is infinite
we set $\N_D(\mathfrak{a}):=0$. Since $D$ will always be clear from the context, we omit the subindex in the norm. The definition perfectly generalizes the norm of an element as $\N((z))=\N(z)$ holds for all $z \in \OD$ (see Lemma~\ref{lem:app:norm}).\\[12pt]
For explaining in which sense ideals in $\OD$ can be decomposed into prime ideals, it is necessary to calculate $\textrm{Spec} \ \OD$ at first. To do this, it turns out to be very useful to consider ideals as $\ZZ$-modules. Like in the case of non-square discriminants, it is essential to see that every $\ZZ$-module in $\OD$ is generated by at most two elements.

\begin{prop} \label{prop:modules}
 Let $M \subset \OD$ be a $\ZZ$-module in $\OD$. Then there exist integers $m,n \in \ZZ_{\geq 0}$ and $a \in \ZZ$
 such that $$M = [n\mathds{1}; a\mathds{1}+mw]:= n\mathds{1}\ZZ \oplus (a\mathds{1}+mw)\ZZ.$$ A non-zero $\ZZ$-module $M=[n;a+mw]$ is an ideal if and only if $m|n$, $m|a$, i.e. $a=mb$ for some $b \in \ZZ$, and $n|m\N(b+w)$.
\end{prop}

\begin{proof} See Proposition~\ref{prop:modules_appendix} and Proposition~\ref{prop:ideal_conditions_appendix}. \end{proof}

As it simplifies notation and cannot cause any confusion we will from now on leave away the symbol $\mathds{1}$ when embedding $\ZZ$ into $\OD$.  

                                                                                                                 


\begin{prop} \label{thm:spec} Let $\OD$ be a quadratic order of square discriminant. Then
\begin{eqnarray*}
 \textrm{Spec} \ \OD & = & \left\{ [p;w], \ [p;d+w] \ | \ p \in \ZZ \ \textrm{prime with } p \nmid D \right\}\\
 & \cup & \left\{ [p;w] \ | \ p \in \ZZ \ \textrm{prime with } p | D \right\}\\
\end{eqnarray*}
\end{prop}

\begin{proof}
 See Theorem~\ref{prop:Spec}.
\end{proof}


Based on this result a ramification theory for prime numbers over $\OD$ can be deduced. Recall that an ideal is called \textbf{irreducible} if it cannot be written as the intersection of two larger ideals.

\begin{thm} \label{thm:ramification_of_primes}
 Let $p \in \ZZ$ be a prime number. 
 \begin{itemize}
 	\item[(i)] If $p \nmid d$ then $(p)=\mathfrak{p}\mathfrak{p}^\sigma$ for a prime ideal $\mathfrak{p}$ of norm 
 	$p$ with $\mathfrak{p} \neq \mathfrak{p}^\sigma$, i.e. $p$ splits.
 	\item[(ii)] If $p|d$ then $(p)$ is an irreducible ideal which is not prime.  
 \end{itemize}  
\end{thm}
  \begin{proof} A detailed proof is given in the appendix \end{proof}
 Note that the main difference in comparison to the non-quadratic case is that divisors of $d$ are not ramified prime numbers.
 
\begin{cor} \label{Sec2:cor:prime_decomposition}
	Every ideal $\mathfrak{a} \subset \OD$ with $\gcd(\N(\mathfrak{a}),d)=1$ can be uniquely written as product of prime ideals.
\end{cor}
\begin{proof} See Appendix, Corollary~\ref{cor:prime_decomposition_appendix}. \end{proof}

Now let $h \in \ZZ$ be an arbitrary integer. Then the principal ideal $(h)$ may be uniquely decomposed as
\begin{align} \label{sec2:eq1}
(h)= \prod_{\substack{p|h, \\ p|d}} p^{e_p} \prod_{\substack{\mathfrak{q}|h, \\ \mathfrak{q} \nmid d}} \mathfrak{q}^{f_\mathfrak{q}} \mathfrak{q}^{\sigma f_\mathfrak{q}},
\end{align}
with $e_i,f_i \in \NN$ where the products are taken over all prime numbers respecitvely prime ideals which satisfy the given conditions.

\paragraph{The special linear group.} We define $\SL_2(\OD)$ to be the group of all~$2$ by~$2$ matrices
with entries in $\OD$ and determinant $1$. In other words, an element
$$A=\begin{pmatrix} (a_1,a_2) & (b_1,b_2) \\ (c_1,c_2) & (d_1,d_2) \end{pmatrix} \in \SL_2(\OD)$$
has determinant $(a_1d_1-b_1c_1,a_2d_2-b_2c_2) = (1,1).$ 
\paragraph{Congruence subgroups.}  For a regular ideal $\mathfrak{a} \subset \OD$ the \textbf{principal congruence subgroup of level $\mathfrak{a}$} is given by
$$\Gamma^D(\mathfrak{a}) := \left\{ \begin{pmatrix} a & b \\ c & d \end{pmatrix} \in \SL_2(\OD): \begin{pmatrix} a & b \\ c & d \end{pmatrix} 
\equiv \begin{pmatrix} 1 & 0 \\ 0 & 1 \end{pmatrix} \mod \mathfrak{a} \right\}.$$
As usual, a subgroup $\Gamma \subset \SL_2(\OD)$ is called a \textbf{congruence subgroup}, if it contains a principal congruence subgroup.
The two most impotant examples of congruence subgroups which we are interested in are
$$\Gamma^D_0(\mathfrak{a}) := \left\{ \begin{pmatrix} a & b \\ c & d \end{pmatrix} | \begin{pmatrix} a & b \\ c & d \end{pmatrix} \equiv \begin{pmatrix} * & * \\ 0 & * \end{pmatrix} \mod \mathfrak{a} \right\} $$
and
$$\Gamma^D_1(\mathfrak{a}) = \left\{ \begin{pmatrix} a & b \\ c & d \end{pmatrix} | \begin{pmatrix} a & b \\ c & d \end{pmatrix} \equiv \begin{pmatrix} 1 & * \\ 0 & 1 \end{pmatrix} \mod \mathfrak{a} \right\}.$$
One might also define $\Gamma^D(\mathfrak{a})$ as the kernel of the projection $\SL_2(\OD) \to \SL_2(\OD/\mathfrak{a})$. Indeed, even the following is true: 
\begin{prop} \label{prop_exact_sequence}
   The sequence 
   $$ 0 \to \Gamma^D(\mathfrak{a}) \to \SL_2(\OD) \to \SL_2(\OD/\mathfrak{a}) \to 0$$
   is exact.
\end{prop}
\begin{proof} See appendix. \end{proof}
\paragraph{The index of some congruence subgroups.} Recall that the index of $\Gamma_0(n)$ in $\SL_2(\ZZ)$ is given by
\begin{align} \label{sec2:eq2}
[\SL_2(\ZZ) : \Gamma_0(n)] = n \prod_{p|n} \left( 1 + \frac{1}{p} \right),
\end{align}
where the product is taken over all prime numbers $p$ dividing $n$. In this paragraph, we calculate analogous formulas for the indexes of $\Gamma^D(\mathfrak{a})$ and $\Gamma^D_0(\mathfrak{a})$ in $\SL_2(\OD)$ for an arbitrary ideal $\mathfrak{a} \subset \OD$. It is possible to imitate the standard proof of the $\SL_2(\ZZ)$ case, which is applied e.g. in \cite[Chapter~2.4]{Kil08}.  
Obviously, the following inclusions hold for any ideal $\mathfrak{a} \subset \OD$:  
$$\Gamma^D(\mathfrak{a}) \subset \Gamma^D_1(\mathfrak{a}) \subset \Gamma^D_0(\mathfrak{a}) \subset \SL_2(\OD).$$
We now state a lemma about the indexes of these inclusions. For the proofs we refer the reader to \cite{Kil08} or \cite{Wei08}. 
\begin{lem} \label{lem:structure_of_congruence_groups} Let $\mathfrak{a} \subset \OD$ be a regular ideal of finite norm.
 \begin{itemize}
  \item[(i)] We have $$[\SL_2(\OD):\Gamma^D_0(\mathfrak{a})]= \# P^1(\OD/\mathfrak{a}),$$
   where $P^1(\cdot)$ denotes the projective space of dimension one. 
  \item[(ii)]We have  $$[\Gamma^D_0(\mathfrak{a}) : \Gamma^D_1(\mathfrak{a})] = \N(\mathfrak{a}).$$
  \item[(iii)]We have $$[\Gamma^D_1(\mathfrak{a}) : \Gamma^D(\mathfrak{a})] = \phi^D(\mathfrak{a})$$
   where $\phi^D(\cdot)$ is the generalized Euler totient function, i.e. it counts the number of units in $\OD/\mathfrak{a}$.
 \end{itemize}
\end{lem}

What is left to do is to calculate the number of elements of the projective space and the number of units. 
\begin{lem} \label{lem:2.11} Let $\mathfrak{a} \subset \OD$ be a regular ideal of finite norm. Then
$$ \phi^D(\mathfrak{a}) = \N(\mathfrak{a}) \prod_{\mathfrak{p}|\mathfrak{a}} (1 - 1/\N(\mathfrak{p}))$$
and
$$ \# P^1(\OD/\mathfrak{a}) = \N(\mathfrak{a}) \prod_{\mathfrak{p}|\mathfrak{a}} (1 + 1/\N(\mathfrak{p})),$$
hold.
\end{lem}
\begin{proof} Let $R$ be any finite ring. As a consequence $R = \prod_{\mathfrak{q}} R_{\mathfrak{q}}$ is a product of local Artinian rings $R_{\mathfrak{q}}$ (see \cite{AM69}, Chapter~8). The $R_{\mathfrak{q}}$ are of the form $\OD / \mathfrak{p}^k$ with $\mathfrak{p}$ as in the formula. Let $\mathfrak{q} \subset R_{\mathfrak{q}}$ be the maximal ideal. Then one has $\OD / \mathfrak{p} \cong R_{\mathfrak{q}} / \mathfrak{q}$ and thus $\# R_{\mathfrak{q}} / \mathfrak{q} = \mathcal{N}(p).$  
The second claim follows by oberserving that
$$\#P^1(R_{\mathfrak{q}}) = 2 \# \mathfrak{q} + (\# R_{\mathfrak{q}}^\times) = \#R_{\mathfrak{q}} \left(1 + \frac{1}{\#(R_{\mathfrak{q}}/\mathfrak{q})} \right).$$
\end{proof}
This enables us to prove the following important statement.
\begin{thm} \label{thm:indexes_of_congruence_subgroups}
	Let $\mathfrak{a} \subset \OD$ be an arbitrary regular ideal.
	\begin{itemize}
		\item[(i)]  The index of $\Gamma^D_0(\mathfrak{a})$ in $\SL_2(\OD)$ is given as
		$$\N(\mathfrak{a}) \prod_{p|\N(\mathfrak{a})} \left( 1 + \frac{1}{p} \right)^{c_p(\mathfrak{a})},$$
		where the product is taken over all prime numbers $p$ and $c_p(\mathfrak{a}) = 2$, if $p \nmid D$ and $\mathfrak{a} \subset (p)$, and $c_p(\mathfrak{a}) = 1$ otherwise.
		\item[(ii)] The index of $\Gamma^D(\mathfrak{a})$ in $\SL_2(\OD)$ is given as 
		$$\N(\mathfrak{a})^3 \prod_{p|\N(\mathfrak{a})} \left( 1 - \frac{1}{p^2} \right)^{c_p(\mathfrak{a})},$$
		where we use the same notation as in $(i)$.
	\end{itemize}
\end{thm}

\begin{proof} By Lemma~\ref{lem:structure_of_congruence_groups} (i) and Lemma~\ref{lem:2.11} we have
	$$[\SL_2(\OD) : \Gamma^D_0(\mathfrak{a})] =  \N(\mathfrak{a}) \prod_{\mathfrak{p}|\mathfrak{a}} (1 + 1/\N(\mathfrak{p})),$$
	where $\mathfrak{p}|\mathfrak{a}$ means $\mathfrak{a} \subset \mathfrak{p}$. If $\mathfrak{p}|\mathfrak{a}$ then $\N(\mathfrak{p}) | \N(\mathfrak{a})$ and $p = \N(\mathfrak{p})$ is a prime number. Conversely, if $p|\N(\mathfrak{a})$, then there are at least one and at most two prime ideals $\mathfrak{p}$ with $\N(\mathfrak{p}) = p$ and $\mathfrak{p}| \mathfrak{a}$. By Theorem~\ref{prop:Spec}, there are two if and only if $c_p(\mathfrak{a}) = 2.$ This proves (i). Assertion (ii) follows similarly by Lemma~\ref{lem:structure_of_congruence_groups} (iii) and Lemma~\ref{lem:2.11}.
	
\end{proof}

In other words, the indexes of $\Gamma^D_0(\mathfrak{a})$ and of $\Gamma^D(\mathfrak{a})$ are as big as possible. 

\paragraph{Non-congruence subgroups.} A finite index subgroup $\Gamma$ of $\SL_2(\OD)$ is a congruence subgroup if and only if there exists an ideal $\mathfrak{a} \subset \OD$ such that the \textbf{level index} $[\SL_2(\OD/\mathfrak{a}) : \rho_{\mathfrak{a}}(\Gamma)]$ is equal to the index $[\SL_2(\OD):\Gamma]$ where $\rho_{\mathfrak{a}}: \SL_2(\OD) \to \SL_2(\OD/\mathfrak{a})$ is the natural projection. The group $\Gamma$ is called a \textbf{non-congruence subgroup of level $\mathfrak{a}$} if the two indices differ and $\Gamma$ is called a \textbf{totally non-congruence subgroup of level $\mathfrak{a}$} if $[\SL_2(\OD/\mathfrak{a}) : \rho_{\mathfrak{a}}(\Gamma)]=1$. Being a totally non-congruence subgroup of level $\mathfrak{a}$ is yet equivalent to the index $[\Gamma:\Gamma \cap \Gamma(\mathfrak{a})]$ being equal to the index  $[\SL_2(\OD):\Gamma(\mathfrak{a})].$ Note that the notion of being a totally non-congruence subgroup of level $\mathfrak{a}$ still makes sense in a situation where $\Gamma$ is of infinite index if this property is defined via 
$[\Gamma:\Gamma \cap \Gamma(\mathfrak{a})]=[\SL_2(\OD):\Gamma(\mathfrak{a})].$ More details on non-congruence subgroups can be found e.g. in \cite[Chapter~3]{WS12}.


\paragraph{Nori's Theorem.} We have just seen that congruence subgroups are closely related to subgroups of $\GL_n(\mathbb{F}_m)$ 
where $\mathbb{F}_m$ denotes the finite field with $m$ elements. In this situation one of the most powerful tools is Nori's theorem.
It describes the subgroups of $\GL_n(\mathbb{F}_p)$ for $n \in \NN$ arbitrary and $p \in \NN$ a prime number with $p>n$. We
closely follow the exposition in \cite[Chapter~3.2]{Rap12} here: Let $H$ be an arbitrary subgroup of $\GL_n(\mathbb{F}_p)$, let 
$X:=\left\{ x \in H \ | \ x^p = 1 \right\}$ and let $H^+= \langle X \rangle \subset H$. An element $x \in H$
lies in $X$ if and only if $(x-1)^n=1$. For fixed $x \in X$ we may thus define
$$\log x:= - \sum_{i=1}^{p-1} \frac{(1-x)^i}{i}.$$
Observing that $\log(x)^n=0$, we see that for any $t \in \overline{\mathbb{F}_p}$, the algebraic closure of $\mathbb{F}_p$,
we can define 
$$x(t):=\exp(t \cdot \log x) \quad \textrm{where} \quad \exp z=\sum_{i=0}^{p-1}\frac{z^i}{i!}.$$
We regard $x(t)$ as a $1$-parameter subgroup of $\GL_n$ and let $\widetilde{H}$ be the $\mathbb{F}_p$-subgroup of $\GL_n$
generated by the $x(t)$. 
\begin{thm} \label{thm:Noris_theorem} \textbf{(Nori, \cite{Nor87})} If $p$ is large enough (for a given n), then $H^+$ coincides with $\widetilde{H}(\mathbb{F}_p)$, the subgroup
of $\widetilde{H}(\mathbb{F}_p)^+$ generated by all unipotents contained in it.  
\end{thm}
If we specify to the case of $\GL_2$, which is the only case of interest here, we get (compare \cite[Chapter~3.2]{Rap12}):
\begin{cor} \label{cor:Nori}
  For any subgroup of $H \subset \GL_2(\mathbb{F}_p)$, the subgroup $H^+$ has either $1$ or $p$ or $p^3-p$ elements.
\end{cor}

\section{Teichmüller curves} \label{Sec:TMC}

A \textbf{Teichmüller curve} \index{Teichmüller curve} $C \to \Mg$ is an algebraic curve in $\Mg$ that is totally geodesic with respect to the Teichmüller metric. Every Teichmüller curve stems from the projection of a $\SL_2(\RR)$-orbit of a translation surface $(X,\omega) \in \Omega\Mg$ to $\Mg$ (see e.g. \cite{Möl11}). The stabilizer of $(X,\omega)$ under the $\SL_2(\RR)$-action is called its \textbf{Veech group}.  On the contrary, the projection of the $\SL_2(\RR)$-orbit of a translation surface $(X,\omega)$ to $\Mg$ yields a Teichmüller curve if and only if its \textbf{Veech group} is a lattice. Moreover the Veech group is never cocompact. Although Teichmüller curves in higher genera $\Mg$ are not satisfactorily well understood yet, there has been great progress on Teichmüller curves in $\mathcal{M}_2$ in the last fifteen years (see e.g. \cite{Bai07}, \cite{McM03}, \cite{Muk11}).\\[12pt]
Translation surfaces can also be characterized by the fact that there exists an atlas on $X$ such that all transition maps (away from the zeroes $Z(\omega)$ of $\omega$) are given by translations: the charts are obtained by integrating $\omega$ locally on $X$ (see \cite{HS06}, Sect. 1.1.3). This gives rise to a distinguished metric on $X$, which is given by pulling back the Euclidean metric on $\CC$ via the charts. A zero of $\omega$ corresponds to a singularity of the metric. A geodesic segment connecting two singularities is called a \textbf{saddle connection}.\\[12pt]
The latter point of view allows another characterization of the Veech group: let $\Aff^+(X,\omega)$ denote the group of orientation-preserving homeomorphisms of $X$ that are affine on $X \setminus Z(\omega)$ with respect to the charts given by integrating $\omega$. The matrix part of the affine map is independent of the charts since transition maps are translations. This provides a map $D: \Aff^+(X,\omega) \to \SL_2(\RR)$ and the image of $D$ is equal to the Veech group of $(X,\omega)$.\\[12pt] 
%
The simplest examples of Teichmüller curves are generated by \textbf{square-tiled surfaces} (or Origamis), i.e. translation 
surfaces $(X,\omega)$, where $X$ is obtained as a covering of a torus ramified over at most one point and $\omega$ is the 
pullback of the holomorphic one-form on the torus. A square-tiled surface is called \textbf{primitive} if the developing vectors of the
saddle connections span $\ZZ^2$. The square-tiled surfaces which we will treat here exclusively are the
$L$-shaped polygons $L(m,n)$ that all lie $\Omega \Mtwo(2)$ (compare Figure $1$).
\begin{center}
\includegraphics[height=3cm]{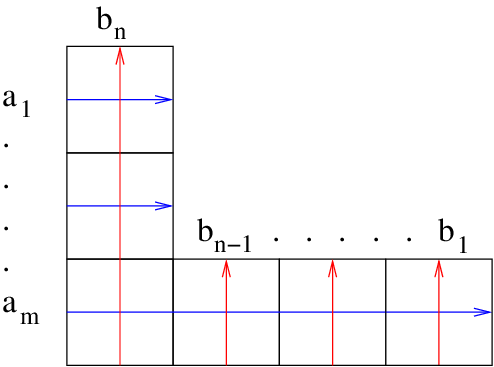}\\
Figure 1. A $L(m,n)$ square-tiled surface with indicated pullback of the canonical symplectic basis of $H_1(E_1,\ZZ)$.
\end{center}
\paragraph{Calculating elements of the Veech group.} Recently, Mukamel designed an algortihm that allows to effectively calculate generators of the Veech group, \cite{Muk13}. For square-tiled surfaces, this task is even easier and described in \cite{WS05}. For our purpose, it is not important to know the complete Veech group but it suffices to calculate some of its parabolic elements. These are obtained by \textbf{cylinder decompositions}: A \textbf{cylinder} of $(X,\omega)$ is a maximal connected set of homotopic simple closed geodesics. Consider a direction of an arbitrary saddle connection of $(X,\omega)$. Then $(X,\omega)$ decomposes into at most two maximal cylinders, bounded above and below by unions of saddle connections. The \textbf{modulus} $m$ of a cylinder is its height divided by its circumference. In fact, the moduli $m_1, m_2$ of the two cylinders of the decomposition must be rationally related for square-tiled surfaces, i.e. there is a number $r \in \QQ$ such that $m_1 r = m_2$. The main ingredient for calculating parabolic elements of the Veech group is the following proposition.
\begin{prop}[Veech, \cite{Vee89}, Proposition 2.4] \label{sec3:prop:veech} If $(X,\omega)$ has a decomposition into cylinders for the horizontal direction such that the moduli of the cylinders have rational ratios, then the Veech group contains the element
	$$T = \begin{pmatrix} 1 & \alpha \\ 0 & 1 \end{pmatrix}$$
where $\alpha$ is the least common multiple of the inverse of the moduli. 
\end{prop}
For a given direction of a saddle connection $(\theta_1,\theta_2) \in \RR^2$ a parabolic element of the Veech group can thus be obtained by rotation. If $\theta_2 \neq 0$, then
\begin{align} \label{sec3:eq1}
 \begin{pmatrix} \theta_1 & 0 \\ \theta_2 & 1 \end{pmatrix} \cdot \begin{pmatrix} 1 & \alpha \\ 0 & 1 \end{pmatrix} \cdot \begin{pmatrix} \theta_1 & 0 \\ \theta_2 & 1 \end{pmatrix}^{-1}
\end{align}
is an element of the Veech group, where the height is measured in vertical direction and $\alpha$ as in Proposition~\ref{sec3:prop:veech}. If $\theta_2 = 0$ then the circumference of the cylinders is measured in vertical direction and their height is measured in horizontal direction such that the Veech group contains
\begin{align} \label{sec3:eq2}
\begin{pmatrix} 0 & 1 \\ 1 & 0 \end{pmatrix} \cdot \begin{pmatrix} 1 & \alpha \\ 0 & 1 \end{pmatrix} \cdot  \begin{pmatrix} 0 & 1 \\ 1 & 0 \end{pmatrix}^{-1}.
\end{align}

\paragraph{The decomposition of homology.}  \label{sec:action_on_homology} A square-tiled surface $\pi_1: X \to E_1$ with $g(X)=2$ and  consisting of $d$ unit squares is called \textbf{minimal} if it does not factor via an isogeny. In this case, it is known that $X$ is 
\textbf{split Jacobian}, i.e. for a minimal square-tiled surface there exists a torus $E_2$ and a covering $\pi_2: X \to E_2$ of the same degree such that  the induced morphism $\Jac(X) \to E_1 \times E_2$ is an isogeny of degree $d^2$ (compare \cite{Kuh88}). We call $E_2$ the \textbf{complementary elliptic curve}. Note that we may identify the $2$-torsion points of $\Jac(X)$ and the $2$-torsion points of $E_1 \times E_2$ if $d$ is odd. In our case, $E_2$ can be chosen in a canonical way. Roughly speaking, it is just the complementary variety of $E_1 \subset \Jac(X)$. An explicit construction of $E_2$ is given e.g. in \cite[Proposition~2.7]{Kan03} (see also \cite[Chapters~5 and 12]{BL04}). Most importantly, $H_1(E_2,\ZZ)$ is the symplectic orthogonal complement of $H_1(E_1,\ZZ)$ inside $H_1(X,\ZZ)$.
\begin{rem}
  A square-tiled surface is minimal if and only if it is primitive (see \cite[Chapter 2]{Kap11}).
\end{rem}

This implies that $H_1(X,\ZZ)$ contains $\Lambda:= H_1(E_1,\ZZ) \oplus H_1(E_2,\ZZ)$
 as a sublattice of index $d^2$ and that the symplectic pairing on $H_1(X,\ZZ)$ respects this decomposition and is of type
$(d)$ on each direct summand. The action of the affine group $\Aff^+(\pi_1)$ induces an action on homology and thus induces an action on
 $H_1(E_i,\ZZ)$. We denote these automorphism groups by $\Gamma_i$ and let $\Gamma_i(\phi)$ be the image of $\phi \in \Aff^+(\pi_1)$.
\begin{lem} If we regard the affine group $\Aff^+(\pi_1)$ as a subgroup of $\SL_2(\ZZ)$ (and not of $\PSL_2(\ZZ)$) then the differential map
$D: \Aff^+ \to \Gamma(\pi_1)$ is an isomorphism. More precisely, $\Gamma_1(\phi)=D(\phi)$. Hence there is a natural homomorphism
$f: \Gamma(\pi_1) \to \Gamma_2$.
\end{lem}
\begin{proof} The first statement follows from the fact that minimal genus~2 covers have no internal automorphisms if $d>2$
\cite[Proposition~2.1]{Kan03}, i.e. the only autmorphism of the square-tiled surface which preserves the covering is the identity map.  Hence $D: \Aff^+(\pi_1) \to \Gamma(\pi_1)$ is an isomorphism. The second statement is just
the definition of the action of the affine group.
\end{proof}
From now on and until the end of this paper we fix the bases of both of the homology groups $H_1(E_i,\ZZ)$. The explicit choice
of the basis on $H_1(E_1,\ZZ)$ is indicated in Figure~1. Kani's result in \cite[Chapter~4 and 5]{Kan03} implies that for this choice of bases the reduction
of $f \mod d$ is conjugation by the diagonal matrix $R:=\diag(-1,1)$. Thus for a given matrix $A$ in the Veech group the induced action on $\Lambda$ is given by a
pair of matrices $\widetilde{A}:=(A_1=A,RA_2R^{-1})$ where $A_i$ acts on $H_1(E_i,\ZZ)$ and $\widetilde{A} \in \SL_2(\OD)$. 
We call the corresponding group the \textbf{homology Veech group}.
\begin{exa} \label{exa:l22} The algorithm described in \cite{WS05} yields that the Veech group of $L(2,2)$ is generated by the matrices
$$\begin{pmatrix} 1 & 0 \\ 2 & 1 \end{pmatrix}, \quad  \begin{pmatrix} 0 & 1 \\ -1 & 0 \end{pmatrix}.$$	
Analyzing their action on homology it follows that
$$\begin{pmatrix} 1 & 0 \\ 2-w & 1 \end{pmatrix}, \quad \begin{pmatrix} 0 & 1 \\ -1 & 0 \end{pmatrix}$$
is a set of generators of the homology Veech group (details on how to do the calculation on homology explicitly are presented on page~\pageref{sec3:calculating:homologyVG}).
\end{exa}
\paragraph{Spin structure.} Due to a result of Kani in \cite{Kan03} the spin structure of a square-tiled surface in $\Omega \Mtwo(2)$ may be defined via the number of integral Weierstra\ss \ points: recall that in genus~$2$ the \textbf{Weierstra\ss \ points} of a square-tiled surface are the six fixed points of the hyperelliptic involution. A Weierstra\ss \ point is called \textbf{integral} if it is a vertex of one of the squares.
\begin{prop} \textbf{(Kani, \cite[Proposition~2.4]{Kan03}, see also \cite{Kuh88})} A primitive square-tiled surface in $\Omega \Mtwo(2)$ consisting 
of $d$ squares has
\begin{itemize}
 \item for $d=3$, exactly 1 integral Weierstra\ss \ point
 \item for even $d$, exactly~$2$ integral Weiersta\ss \ points 
 \item for odd $d>3$, either $1$ or $3$ integral Weierstra\ss \ points and both values occur.
\end{itemize}
\end{prop}
If $d$ is odd, the \textbf{spin structure} of the square-tiled surface is called even if the number of Weierstra\ss \ point is $1$ and
otherwise it is called odd. The notion of spin structure we use here is a special case of a more general concept, which for an arbitrary Riemann surface $X \in \mathcal{M}_g$ corresponds to the choice of a square-root of the canonical line bundle up to isomorphism, see e.g. \cite[Chapter 6]{McM05}.
\begin{rem}
  The result of Kani was originally formulated in the language of abstract algebraic geometry. Its relevance for
  square-tiled surfaces was first observed in \cite[Remark~3.4]{Möl05}. Our formulation of the result goes back to 
  \cite[Proposition~4.3]{HL06}.
\end{rem}
Hubert and  Leli\`{e}vre in \cite[Theorem 1.1]{HL06} and more generally McMullen in \cite[Corollary 1.2]{McM05} proved that the $\SL_2(\RR)$-orbits of square-tiled surfaces in $\Omega \Mtwo(2)$ can be distinguished by their spin structure. We restate their result in a way which is more applicable for us (compare also \cite[Theorem B]{WS12}):

\begin{thm} \label{cor:classification} \textbf{(Hubert/Leli\`evre/McMullen)} The set of primitive square-tiled surfaces in $\Omega\Mtwo(2)$ with $d$ squares forms one single
$\SL_2(\ZZ)$ orbit, if $d$ is even or $d=3$. They form two orbits called $\mathcal{A}_d$ (corresponding to even spin) and $\mathcal{B}_d$ (corresponding to odd spin) distinguished
by their number of integral Weierstra\ss \ points, if $d$ is odd and greater than $3$.
A square-tiled surface $L(m,n)$ with $d=m+n-1$ squares belongs to $\mathcal{A}_d$ if both $m$ and $n$ are even and belongs to $\mathcal{B}_d$ if both
$m$ and $n$ are odd. Each such square-tiled surface-orbit is generated by some $L(m,n)$.
\end{thm}
\begin{rem}
  More generally, Teichmüller curves in $\Omega \Mtwo(2)$ have been completely classified by McMullen in \cite{McM05}
  by their spin structure and their discriminant.
\end{rem}
We embed the translation surface $X \in \Mtwo(2)$ generating the Teichm\"uller curve in its Jacobian in the following way: we choose an arbitrary Weierstra\ss \ point $z \in X$ and define
$$\varphi_z: X \to \Jac(X), \quad x \mapsto [x-z].$$

\paragraph{Theta functions and theta characteristic.} The $2$-torsion points on an elliptic curve are in natural correspondence with theta characteristics \cite[Corollary~VI.1.5]{FK92}. Moreover the Weiterstra\ss \ points on a
curve of genus~$2$ correspond to the odd theta characteristics (see e.g. \cite[Chapter~VII.1]{FK92}). We will therefore 
quickly introduce this concept, but restrict here to dimension~2: 
Let $\HH_2$ denote the Siegel upper half space of genus~$2$. we define for $\epsilon,\epsilon' \in \left\{ 0, 1 \right\}^{2}$ the \textbf{theta function with theta 
characteristic $(\epsilon,\epsilon')$} on $\CC^2 \times \HH_2$ by
$$\theta \begin{bmatrix} \epsilon \\ \epsilon' \end{bmatrix} (u,Z) := \sum_{x \in \ZZ^2 + \epsilon/2} \exp \left( 2 \pi i \left( 
\frac{1}{2} xZ x^T + x \left( u + \frac{\epsilon'}{2} \right) \right) \right).$$ 
The theta characteristic is called \textbf{odd} if $\epsilon(\epsilon')^T$ is odd and \textbf{even} otherwise. 
Accordingly we denote the zero-locus of the theta function by
$$\Theta := \left\{ z \in \CC^2 \ | \ \theta \begin{bmatrix} \epsilon \\ \epsilon' \end{bmatrix}(z,\Pi) = 0 \right\}$$
where $\Pi$ is a fixed element in $\HH_2$. We leave
away $\Pi$ in the definition of $\Theta$ as it is clear from the context which matrix is meant. With these conventions $\varphi_z(X) = \Theta$ holds if $[z] = \frac{1}{2}(\Id \epsilon' + \Pi \epsilon)$ (compare \cite[Chapter~VII.1.2]{FK92}). The relation of theta functions and Teichmüller curves is discussed in detail in \cite{MZ16}.

\paragraph{Spin structure of the complementary elliptic curve.} By counting the number of integral Weierstra\ss \ points also the 
complementary elliptic curve may be given a spin structure in a canonical way if $d>3$ is odd. We fix
$E_1$ by choosing the ramification point of the covering as $p=(0,0)$. By construction, also the complementary elliptic 
curve $E_2$ is then fixed (it is the symplectic orthogonal complement of $E_1$ inside the Jacobian). The 
Weierstra\ss \ points of the square-tiled surface $\pi : L_d \to E_1$ are preimages of the $2$-torsion points on $E_1$ and 
$1$ or $3$ of them are integral, i.e. lie over $p$.\\[12pt]
Since $\Jac(X) \to E_1 \times E_2$ is an isogeny of degree $d^2$ where $d^2$ is odd all Weierstra\ss \ points of 
$X$ are $2$-torsion points of $E_1 \times E_2$. There is a one-to-one correspondence between 
the Weierstra\ss \ points of $X$ and odd theta characteristics. The latter are
$$\begin{bmatrix} 0 & 1 \\ 0 & 1 \end{bmatrix} \quad \begin{bmatrix} 0 & 1 \\ 1 & 1 \end{bmatrix} \quad \begin{bmatrix} 1 & 1 \\ 0 & 1 \end{bmatrix} \quad 
\begin{bmatrix} 1 & 0 \\ 1 & 0 \end{bmatrix} \quad \begin{bmatrix} 1 & 0 \\ 1 & 1 \end{bmatrix} \quad \begin{bmatrix} 1 & 1 \\ 1 & 0 \end{bmatrix}.$$
More precisely, if we use the fact that $\Jac(X)$ is isogenous to $E_1 \times E_2$ of odd degree $d^2$, then by renormalizing the odd theta characteristics by an odd translation the correspondence can be made as follows (compare \cite[Chapter~VI.3]{FK92}): 
the first column of the theta characteristic divided by~$2$ corresponds to the coordinates of the projection of the Weierstra\ss \ points of $L_d$ to $E_1$ and the second column of the theta characteristic divided by~$2$ corresponds to the coordinates of the  projection of the Weierstra\ss \ points of $L_d$ to $E_2$.

\begin{lem} \label{lem:translation} If the odd theta characteristics are translated by an odd theta
characteristic 
\begin{itemize}
 \item with first column $\begin{bsmallmatrix} 1 \\ 1 \end{bsmallmatrix}$ then there is one translated characteristic with second column $\begin{bsmallmatrix} 0 \\ 0 \end{bsmallmatrix}$.
 \item with first column $\begin{bsmallmatrix} 1 \\ 0 \end{bsmallmatrix}$ or  $\begin{bsmallmatrix} 0 \\ 1 \end{bsmallmatrix}$ or  $\begin{bsmallmatrix} 0 \\ 0 \end{bsmallmatrix}$ then
 there are three translated characteristics with second column $\begin{bsmallmatrix} 0 \\ 0 \end{bsmallmatrix}$.
\end{itemize}
\end{lem}
\begin{proof} This is an immediate calculation. \end{proof}

\begin{thm} \label{thm:weierstrass} Let $\pi_1 : L_d \to E_1$ be a square-tiled surface in $\Omega \Mtwo(2)$ 
and let $E_2$ denote the complementary elliptic curve. If $\pi_1 : L_d \to  E_1$ has even spin, then $\pi_2 : L_d \to  E_2$ has odd spin. If $\pi_1 : L_d \to E_1$ has odd spin, then $\pi_2 : L_d \to E_2$ has even spin.
\end{thm}
\begin{proof}
If $\pi_1 : L_d \to E_1$ has three $3$ integral Weiterstra\ss \ points, we have to renormalize the odd theta characteristics 
by adding a odd theta characteristic with first column $\begin{bsmallmatrix} 1 \\ 1 \end{bsmallmatrix}$. Thus $\pi_2 : L_d \to E_2$ has exactly $1$ integral Weierstra\ss \ point by Lemma~\ref{lem:translation}. If $\pi_1 : L_d \to E_1$ has one integral Weierstra\ss \ point then 
we have to renormalize the odd theta characteristics by adding an odd characteristic 
with first column $\neq \begin{bsmallmatrix} 1 \\ 1 \end{bsmallmatrix}$ to each element of the list (although we do, of course, not know precisely which one). Consequently, $\pi_2 : L_d \to E_2$ has three integral Weierstra\ss \ points again by Lemma~\ref{lem:translation}. 
\end{proof}

\paragraph{Calculating elements of the homology Veech group.} \label{sec3:calculating:homologyVG} Having described how the Veech group acts on homology and how this action yields matrices in $\SL_2(\OD)$, we now calculate some explicit elements of the homology Veech group of $L(m,n)$ with $m+n-1=d$. We denote its Veech group by $\SL(L(m,n))$ and accordingly its homology Veech group by $\SL^1(L(m,n))$. Let $a,b$ be the symplectic basis of $H_1(E_1,\ZZ)$ indicated in Figure~1 and consider its pullback to $H_1(L(m,n),\ZZ)$.
Note that the classes $a_1,\ldots,a_{m-1}$ and the classes $b_1,\ldots,b_{n-1}$ respectively each define only a single class in $H_1(L(m,n),\ZZ)$.
Therefore the pullback of the symplectic basis of $H_1(E_i,\ZZ)$ yields the elements
$$c_1=(m-1)a_1+a_m \ \ \textrm{and} \ \ d_1=(n-1)b_1+b_n.$$
of $H_1(L(m,n),\ZZ)$. For their symplectic pairing we have
$$(c_1,d_1)=(m-1)(n-1)\cdot 0+(m-1) \cdot 1+(n-1)\cdot 1+1=d.$$
The set $c_1,d_1$ may be extended to a symplectic basis of $\Lambda$ by choosing $c_2:=na_1-a_m$ and $d_2:=-mb_1+b_n$ as
$$(c_1,c_2)=(n-1) + (-n) + 1 = 0 \ \ \textrm{and} \ \ (d_1,d_2) = (m-1) + (-m) + 1 = 0$$
and
$$(c_2,d_2)=n+m-1=d.$$
Before we start with the calculation of some elements of the homology Veech group, let us fix the notation for three special matrices in $\SL_2(\ZZ)$ first, namely
$$T':=\begin{pmatrix} 1 & 1 \\ 0 & 1 \end{pmatrix}, \ Z' := \begin{pmatrix} 1 & 0 \\ 1 & 1 \end{pmatrix} \ \ \textrm{and} \ \ S' = \begin{pmatrix} 0 & -1 \\ 1 & 0 \end{pmatrix}.$$ Recall that the Veech group of $L(m,n)$ always contains the element $T'^n$ (Proposition~\ref{sec3:prop:veech}). This matrix acts on the $a_i$ and $b_i$ by
\begin{align*}
a_i & \mapsto a_i \ \ i=1,\ldots,m \\
b_1 & \mapsto b_1 + a_m \\
b_n & \mapsto b_n + a_m + n(m-1)a_1
\end{align*}
since the matrix $T'^n$ yields a single Dehn-twist on the lower cylinder and a $n$-fold Dehn-twist on the upper cylinder and 
since the $a_i$ are parallel to the twist direction. Therefore
$$c_2 \mapsto c_2 \ \ \textrm{and} \ \ d_2 \mapsto (m-1)c_2 + d_2.$$
The action of the Veech group on the homology of $L(m,n)$ hence yields the element $T:=(T'^n,T'^{n-d}) \in \SL_2(\OD)$. Analogously, 
$Z:=(Z'^m,Z'^{m-d}) \in \SL_2(\OD)$ is an element of the homology Veech group by \eqref{sec3:eq2}.\\
Further explicit elements of the Veech group of $L(m,n)$ can be gained by using cylinder decomposition in certain directions and applying \eqref{sec3:eq1}. Consider the cylinder decomposition of $L(m,n)$ in direction $(1,1)$. Only one cylinder occurs here with circumference $m+n-1=d$ and height $1$. Thus their modulus is $\frac{1}{d}$. Proposition~\ref{sec3:prop:veech} and \eqref{sec3:eq1} yield that
\begin{align*}
E' & = \begin{pmatrix} 1 & 0 \\ 1 & 1 \end{pmatrix} \begin{pmatrix} 1 & d \\ 0 & 1 \end{pmatrix} \begin{pmatrix} 1 & 0 \\ 1 & 1 \end{pmatrix}^{-1}\\
& = \begin{pmatrix} 1 - d & d\\-d & 1 + d \end{pmatrix}.
\end{align*}
is an element of the Veech group of $L(m,n)$. Analyzing the action of $E'$ on the homology basis $a_i, b_i$ we get 
\begin{align*}
a_1 & \mapsto ma_1 + a_m + (n-1)b_1 + b_n\\
a_m & \mapsto (m-1)na_1 + (n+1)a_m + (n-1)nb_1 + nb_n\\
b_1 & \mapsto -((m-1)a_1 + a_m + (n-2)b_1 +b_n)\\
b_n & \mapsto -((m-1)ma_1+ma_m + (n-1)mb_1 + (m-1)b_m).
\end{align*}
This implies that the homology Veech group contains the matrix
$$E:=(E',\Id) = \begin{pmatrix} 1 - (d-w) & d-w\\ -(d-w) & 1 +(d-w) \end{pmatrix}.$$
If $n=3$ and $m$ is odd, another element in the homology Veech group can be similarly computed from the cylinder decomposition in direction $(2/m,1)$, namely
\begin{eqnarray*}
F & := &\begin{pmatrix} (1-2(d-2),1+(4-d)) & (4,4-d) \\  (-(d-2)^2,-(4-d)) & (1+2(d-2),1-(4-d)) \end{pmatrix} \\
& = & \begin{pmatrix} 1-(2(d-2)-w) & 4 - w\\ -(d-2)^2+(d-3)w & 1+(2(d-2)+w) \end{pmatrix}
\end{eqnarray*}
Finally, let us specify to the case $n=2$. Then the cylinder decomposition in direction $(2/m,1)$ yields the following elements
in the homology Veech group:
\begin{itemize}
\item If $m \equiv 2 \mod 4$: 
$$F:=\begin{pmatrix} (1-3/2m,2-m/2) & (3,-(m-2)) \\  (-3/4m^2,(m-2)/4) & (1+3/2m,m/2) \end{pmatrix} = \begin{pmatrix} * & 3 - w\\ * & * \end{pmatrix}$$
\item If $m \equiv 0\ \mod 4$: 
$$F:=\begin{pmatrix} (1-3m,3-m) & (6,-2(m-2)) \\  (-3/2m^2,(m-2)/2) & (1+3m,m-1) \end{pmatrix} = \begin{pmatrix} * & 2(3 - w)\\ * & * \end{pmatrix}$$
\item If $m \equiv 1 \mod 2$:
$$F:=\begin{pmatrix} (1-6m,5-2m) & (12,-4(m-2)) \\  (-3m^2,m-2) & (1+6m,2m-3) \end{pmatrix} = \begin{pmatrix} * & 4(3 - w)\\ * & * \end{pmatrix}$$
\end{itemize}

\section{Proof of the main result.}

In this section, we calculate the index of $\Gamma^D_0(\mathfrak{a}) \cap \SL^1(L_d)$ in $\SL^1(L_d)$ for an arbitrary 
ideal $\mathfrak{a} \subset \OD$ and thereby show our main theorem.
\begin{thm} \label{thm:main_thm}
  Let $L_d$ be a square-tiled surface in $\Omega\Mtwo(2)$ with $d$ squares and let $\SL^1(L_d)$ be its homology Veech group. We distinguish the two different cases that $L_d$ is in the orbit $\mathcal{A}_d$ and $\mathcal{B}_d$ in the classification of square-tiled surfaces in $\Omega \Mtwo(2)$. Moreover if $d$ is odd let $2=\mathfrak{p}_2\mathfrak{p}_2^\sigma$ 
  be the decomposition of~$2$ into prime ideals, where $\mathfrak{p}_2$ is the distinguished prime ideal that is a common divisor of~$2$
   and $2-w$ in $\mathcal{O}_{d^2}$. 
  \begin{itemize}
   \item[(1A)] If $d$ is odd and $L_d$ is in $\mathcal{A}_d$, or $d=3$, then
   $[\SL^1(L_d):\SL^1(L_d) \cap \Gamma_0(\mathfrak{a})] = [\SL_2(\mathcal{O}_{d^2}):\Gamma_0(\mathfrak{a})]$ for all ideals $\mathfrak{a} \subset \mathcal{O}_{d^2}$ with
   $\gcd(\mathfrak{p}_2,\mathfrak{a})=1$.
   \item[(1B)] If $d$ is odd and $L_d$ is in $\mathcal{A}_d$, or $d=3$, then $[\SL^1(L_d):\SL^1(L_d) \cap \Gamma_0(\mathfrak{a})] = \frac{2}{3}[\SL_2(\mathcal{O}_{d^2}):\Gamma_0(\mathfrak{a})]$
   for all ideals with $\mathfrak{p}_2|\mathfrak{a}$.
   \item[(2A)] If $d$ is odd and $L_d$ is in $\mathcal{B}_d$, then we have
   $[\SL^1(L_d):\SL^1(L_d) \cap \Gamma_0(\mathfrak{a})] = [\SL_2(\mathcal{O}_{d^2}):\Gamma_0(\mathfrak{a})]$ for all ideals $\mathfrak{a} \subset \mathcal{O}_{d^2}$ with
   $\gcd(\mathfrak{p}_2^\sigma,\mathfrak{a})=1$.
   \item[(2B)] If $d$ is odd and $L_d$ is in $\mathcal{B}_d$, then we have $[\SL^1(L_d):\SL^1(L_d) \cap \Gamma_0(\mathfrak{a})] = \frac{2}{3}[\SL_2(\mathcal{O}_{d^2}):\Gamma_0(\mathfrak{a})]$
   for all ideals with $\mathfrak{p}_2^\sigma|\mathfrak{a}$.
   \item[(3A)] If $d$ is even, then 
   $[\SL^1(L_d):\SL^1(L_d) \cap \Gamma_0(\mathfrak{a})] = [\SL_2(\mathcal{O}_{d^2}):\Gamma_0(\mathfrak{a})]$ for all ideals $\mathfrak{a} \subset \mathcal{O}_{d^2}$ with
   $2 \nmid \N(\mathfrak{a})$.
   \item[(3B)] If $d$ is even, then $[\SL^1(L_d):\SL^1(L_d) \cap \Gamma_0(\mathfrak{a})] = \frac{2}{3}[\SL_2(\mathcal{O}_{d^2}):\Gamma_0(\mathfrak{a})]$
   for all ideals with $2|\N(\mathfrak{a})$.
   \end{itemize}
\end{thm}
It suffices to restrict the analysis to the case $L_d = L(m,n)$. The advantage of doing so is that we can make explicit use of the elements of the homology Veech group found in Section~\ref{Sec:TMC}.\\[12pt]
Moreover, it is convenient to start the analysis with $\mathfrak{a}=(h)$ for $h \in \ZZ$. As every ideal $\mathfrak{a} \subset \OD$ contains a principal ideal generated by an element $h \in \ZZ$, namely e.g. $\N(\mathfrak{a})$, the only ideals preventing us from immediately starting with some $(h)$ are the prime ideals $\mathfrak{p}$ of the form $\mathfrak{p} = [p,w]$. Because
$$[p,w]^2 =[p^2,pw] = (p)[p,w]$$
holds, it follows that
$$[p,w] \supset (p) \supset (p)[p,w] =[p,w]^2 \supset (p^2).$$
Hence, if the index of $[SL_2(\OD):\Gamma^D_0(p^k)]$ was maximal for all $k \in \NN$, also the index of $\Gamma^D_0([p,w]^k)$ would be maximal. In other words, we may indeed restrict to the case $(h)$ with $h \in \ZZ$.\\[12pt]
By Theorem~\ref{thm:ramification_of_primes} and Corollary~\ref{Sec2:cor:prime_decomposition} the principal ideal $(h)$ may be uniquely decomposed as
$$(h)= \prod_{\substack{p|h, \\ p|d}} p^{e_p} \prod_{\substack{\mathfrak{q}|h, \\ \mathfrak{q} \nmid d}} \mathfrak{q}^{f_\mathfrak{q}} \mathfrak{q}^{\sigma f_\mathfrak{q}}$$ 
with $e_i,f_i \in \NN$ where the products are taken over all prime numbers respecitvely prime ideals which satisfy the given conditions. Furthermore we set $H:=\N(h)$.\\[12pt]
We will prove Theorem~\ref{thm:main_thm} step by step: As the index of $\SL^1(L_d) \cap \Gamma^D_0(h)$ in $\SL^1(L_d)$ does not 
depend on the ordering of the divisors of $(h)$ which we choose, we may first divide out the prime number divisors of the discriminant (irreducible but not prime) and afterwards the prime divisors of prime numbers which split (compare Theorem~\ref{thm:ramification_of_primes}). Moreover we may always assume that we consider the 
divisor $\mathfrak{p}$ of $(h)$ which has the highest order in $(h)$ of all divisors $\mathfrak{p}_i$ of the given type.\\[12pt]
To prove the theorem we will proceed in the following way: We first calculate for an arbitrary prime ideal $\mathfrak{p} \subset \OD$ 
and with $\gcd((h),\mathfrak{p})=1$ the index 
$$[(\SL^1(L_d) \cap \Gamma^D_0((h)\mathfrak{p}^k):(\SL^1(L_d)\cap \Gamma^D_0((h)\mathfrak{p}^{k+1}))]$$ 
for all $k \in \mathbb{N}$ and then the index
 $$[(\SL^1(L_d) \cap \Gamma^D((h))):(\SL^1(L_d) \cap \Gamma^D_0((h)\mathfrak{p}))].$$ 
Of course, these equations prove the theorem. We will from now leave away brackets indicating ideals
since it facilitates notation. 
\begin{rem}
	For $d=3$ the claim follows from a direct calculation: recall from Example~\ref{exa:l22} that 
		$$Z:= \begin{pmatrix} 1 & 0 \\ 2-w & 1 \end{pmatrix}, \qquad S:= \begin{pmatrix} 0 & -1\\ 1 & 0 \end{pmatrix}$$
	generate $\SL^1(L_3)$. The matrices $Z$ and $S$ are inequvialent modulo $\Gamma_0^D(\mathfrak{p}_2)$ but there does not exist a matrix not being equivalent to either $Z$ or $S$ modulo $\Gamma_0^D(\mathfrak{p}_2)$. This proves (1B). As an example, how to proceed for (1A), let us consider an arbitrary prime ideal $\mathfrak{p}$ with $\gcd(\N(\mathfrak{p}),6) = 1$ or $\mathfrak{p} = \mathfrak{p}_2^\sigma$: the matrices $Z^{pi}$ for $i=1,\ldots,p$ all lie in $\SL_1(L_d) \cap \Gamma_0^D(\mathfrak{p})$ but are incongruent modulo $\Gamma_0^D(\mathfrak{p}^2)$. Furthermore, the matrices $S$ and $Z^i$ for $i=1,\ldots,p$ all lie in $\SL^1(L_d)$ but are incongruent modulo $\Gamma_0^D(\mathfrak{p})$. For the general case, similar arguments apply as for the other discriminants. The details will be explained in the following. 
\end{rem}
As a shortcut we write
$$T:= \begin{pmatrix} 1 & \eta^+ \\ 0 & 1 \end{pmatrix} \ \ \textrm{and} \ \ Z:=\begin{pmatrix} 1 & 0 \\ \eta^- & 1 \end{pmatrix},$$
where $\eta^+ = n-w$ and $\eta^- = m -w$, because these two matrices are elements of the homology Veech group $\SL^1(L(m,n))$.
\begin{rem} \label{rem:etaproduct}
  For all square-tiled surfaces $L(m.n)$ we have 
   $$\eta^*:=\eta^+\eta^- = (n-w)(m-w) = mn - w.$$
\end{rem}



\paragraph{Divisors of the discriminant.} We begin with the case which is the most special compared to \cite{Wei12} 
since prime numbers $p \in \ZZ$ with $p|d$ are irreducible but not prmie. The main difference to the proof of \cite[Theorem~5.1]{Wei12} is that we use Weitze-Schmithüsen's result at some point (Proposition~\ref{prop:divisors_of_discriminant_2}). 

\begin{prop} Let $p \in \ZZ$ be a prime number with $p|d$ and $\gcd(m,p) = 1$ and let $h \in \ZZ$ be an arbitrary element with $\gcd(h,p)$ $=1$. 
Then for all $k \in \mathbb{N}$ 
$$\left[ (\SL^1(L_d) \cap \Gamma^D_0(hp^k)) :  (\SL^1(L_d) \cap \Gamma^D_0(hp^{k+1})) \right] = \N(p) = p^2$$	
holds.

\end{prop} 

\begin{proof} 
We need to find matrix $W$ which is a word in $T$ and $Z$ such that the matrices
$$W^lZ^{hp^k j}, \quad j=1,\ldots,p, \quad l=1,\ldots,p$$
lie in $\Gamma^D_0(hp^k)$ but are all incongruent modulo $\Gamma^D_0(hp^{k+1})$. The matrix
$$W:= ZT^{hp^k}Z^{-1}$$ 
is a good a choice for it. First, it is in $\Gamma^D_0(hp^{k})$. Secondly
$$W^yZ^{hp^k j} \equiv W^lZ^{hp^k i}$$
is equivalent to $W^yZ^{hp^k (j-i)} W^{-l} \in \Gamma^D_0(hp^{k+1})$.
Setting $x:=j-i$ we thus check when
\begin{eqnarray*} (W^yZ^{hp^k x} W^{-l})_{2,1} & = & \underbrace{p^{3k}h^3\eta^{-3}\eta^{+2}\cdot ylx}_{v_1}\\ & + & \underbrace{p^{2k}h^2\eta^{-2}\eta^+\cdot (y+l)x}_{v_2}\\ & + & \underbrace{p^kh\eta^-(x+\eta^-\eta^+\cdot(l-y))}_{v_3} 
\end{eqnarray*}
is divisible by $hp^{k+1}$. We already know that $hp^{k+1} | v_1+v_2$. 
So we are interested in which cases we have $hp^{k+1} | hp^k\eta^-(x+\eta^-\eta^+\cdot (l-y))$.
By Lemma~\ref{lem:arithmetic_properties} it suffices to check when $p| \eta^-(x+\eta^-\eta^+\cdot(l-y))$ holds.
We have $\eta^-\eta^+ = mn - w$, by Remark~\ref{rem:etaproduct}. Thus we ask when $p|(m-w)((mn-w)\cdot(l-y)+x)$. 
We now just look at the \textit{real part} of the right hand side (i.e. the part in $\ZZ$). 
This gives us that $p|x+mn(l-y)$ since $p \nmid m$. Considering the \textit{imaginary part} (i.e. the coefficient of $w$)
we get $p|m(l-y)$. Thus $l=y$ and therefore $p|x$. This yields $x=0$ or in other words $i=j$.
\end{proof}

Since the proof did not depend on the condition $k>0$ we get an immediate corollary.

\begin{cor} \label{cor:lower_bound}
   Let $p \in \ZZ$ be a prime number with $p|d$ and $\gcd(m,p)=1$ and let $h \in \ZZ$ be an arbitrary element with $\gcd(h,p)=1$. Then 

$$\left[ (\SL^1(L_d) \cap \Gamma^D_0(h)) :  (\SL^1(L_d) \cap \Gamma^D_0(hp)) \right] \geq p^2$$	
holds.
\end{cor}

Before proving that this index is indeed as big as possible, i.e. equal to $p(p+1)$, recall that the homology Veech group $\SL^1(L_d)$ is given by pairs $(A_i,B_i)$ with $A_i$ in the
ordinary Veech group $\SL(L_d)$ and $B_i$ in the so-called \textbf{complementary Veech group $\SL^c(L_d)$}. 

\begin{lem}
If two elements $(A_1,B_1)$ and $(A_2,B_2)$ in $\SL^1(L_d)$ yield the same representative in $\SL^1(L_d)/(\SL^1(L_d) \cap \Gamma_0((h_1,h_2)))$ then $A_1$ and $A_2$ yield the same representative in $\SL(L_d)\/(\SL(L_d) \cap \Gamma_0(h_1))$ and $B_1$ and $B_2$ yield the same representative in $\SL^c(L_d)/(\SL^c(L_d) \cap \Gamma_0(h_2))$. 
\end{lem}

\begin{proof} This is clear by definition. \end{proof}

Now we can easily prove the following: 



\begin{prop} \label{prop:divisors_of_discriminant_2}
   Let $p \in \ZZ$ be a prime number with $p|d$ and $\gcd(2m,p)=1$ and let $h \in \ZZ$ be an arbitrary element with $\gcd(h,p)=1$. Then 
$$\left[ (\SL^1(L_d) \cap \Gamma^D_0(h)) :  (\SL^1(L_d) \cap \Gamma^D_0(hp)) \right] = p(p+1)$$	
holds.
\end{prop}

\begin{proof}
  Let $q$ be the index $\left[ (\SL^1(L_d) \cap \Gamma^D_0(h)) :  (\SL^1(L_d) \cap \Gamma^D_0(hp)) \right]$ and let $h=(h_1,h_2)$. We know by
  Corollary~\ref{cor:lower_bound} that $q\geq p^2$. Furthermore 
  $$\left[ (\SL^1(L_d) \cap \Gamma^D_0(h)) :  (\SL^1(L_d) \cap \Gamma^D_0(hp)) \right]$$ 
  is divided by
  $$\left[ (\SL(L_d) \cap \Gamma_0(h)) :  (\SL(L_d) \cap \Gamma_0(h_1p)) \right].$$ 
  Finally, it follows that
  \begin{align*}
  \left[ (\SL(L_d) \cap \Gamma_0(h)) : \right. & \left. (\SL(L_d) \cap \Gamma_0(h_1p)) \right]\\ &  = \frac{\left[ (\SL(L_d) :  (\SL(L_d) \cap \Gamma_0(h_1p)) \right]}{\left[ \SL(L_d) :  (\SL(L_d) \cap \Gamma_0(h_1)) \right]}\\ & \stackrel{\textrm{Corollary~\ref{cor:schmitt}, Case $2 \nmid p$}}{=} \frac{[\SL_2(\ZZ):\Gamma_0(h_1p)]}{\SL_2(\ZZ):\Gamma_0(h_1)}\stackrel{\eqref{sec2:eq2}}{=}p+1.
  \end{align*}
\end{proof}

\paragraph{Divisors of split prime numbers.} The second case which we treat concerns prime numbers $p \in \ZZ$ that split,
i.e. $p \in \ZZ$ with $p \nmid d$. Choose the prime ideal $\mathfrak{p}$ such that $p=\mathfrak{p}\mathfrak{p}^\sigma$.

\begin{lem} \label{lem:split_divisors_1}
  Assume $\gcd(\mathfrak{p},\eta^-)$ $=1$ and let $h \subset \OD$ be an arbitrary ideal with $\gcd(h,\mathfrak{p})=1$. Then for all $k \in \mathbb{N}$ 
$$\left[ (\SL^1(L_d) \cap \Gamma^D_0(h\mathfrak{p}^k)) :  (\SL^1(L_d) \cap \Gamma^D_0(h\mathfrak{p}^{k+1})) \right] = p $$	
holds.
\end{lem}
   
\begin{proof} The aim is to find $p$ matrices in $\Gamma^D_0(h\mathfrak{p}^k)$ which are inequivalent modulo $\Gamma^D_0(h\mathfrak{p}^{k+1})$. 
We will now describe the simplest set of matrices which we found. We have to distinguish two cases.\\[12pt]
\textit{1. Case: $\mathfrak{p}$ is not conjugated to any of the $\mathfrak{q}$ dividing $h$, i.e. $\mathfrak{q} \neq \mathfrak{p}^\sigma$ for all prime ideals $\mathfrak{q}|h$.} \\
Then $Z^{Hp^ki}$, $i=1,\ldots,p$, where $H= \mathcal{N}(h)$, are obviously in $\Gamma^D_0(h\mathfrak{p}^k)$ but inequivalent modulo $\Gamma^D_0(h\mathfrak{p}^{k+1})$. \\[12pt]
\textit{2. Case: $\mathfrak{p}$ is conjugated to a certain $\mathfrak{q}$ dividing $h$ with order $f_{\mathfrak{q}}$}\\
We set $\mathfrak{h}:=h\mathfrak{q}^{-f_\mathfrak{q}}$ and $H':=\N(\mathfrak{h})$. Let us assume that $\mathfrak{q}^l|\eta^-$ and
$\mathfrak{q}^{l+1} \nmid \eta^-$ for some $l \in \ZZ_{\geq 0}$. We now have to distinguish two subcases.\\
\textit{Case (a)} $f_\mathfrak{q} - l > k$\\ 
In particular, this implies that $\mathfrak{q}$ has a higher order in $h$ than $k$. We therefore want to divide out powers of $\mathfrak{q}$ first. This means that we have to show 
$$ \left[ (\SL^1(L_d) \cap \Gamma^D_0(\mathfrak{h}\mathfrak{q}^{f_\mathfrak{q}-1}\mathfrak{p}^k)) :  (\SL^1(L_d) \cap  \Gamma^D_0(\mathfrak{h}\mathfrak{q}^{f_\mathfrak{q}}\mathfrak{p}^k)) \right] = p.$$
We set $u:=f_\mathfrak{q}-l-1$ and $v:=p^u$. The matrices $Z^{H'vi}$, $1 \leq i \leq p$ lie in $\Gamma^D_0(\mathfrak{h}\mathfrak{q}_h^{f_q-1}\mathfrak{p}^k))$
since $\mathfrak{h}\mathfrak{q}^{f_\mathfrak{q}-1}\mathfrak{p}^k| H'v\eta^- \cdot i$ but the matrices are incongruent modulo 
$\Gamma^D_0(\mathfrak{h}\mathfrak{q}^{f_\mathfrak{q}}\mathfrak{p}^k)$ since $\mathfrak{h}\mathfrak{q}^{f_\mathfrak{q}}\mathfrak{p}^k| H'v\eta^- \cdot i$ implies $\mathfrak{q}|i$ by definition of $v$. This means that we may restrict to case $(b)$, namely:\\ 
\textit{Case (b)} $k \geq f_\mathfrak{q} - l$.\\
We set $v = p^k$ and look at the matrices $Z^{H'vi}$, $1 \leq i \leq p$. These matrices are all in $\Gamma^D_0(h\mathfrak{p}^k)$ but are not equivalent modulo $\Gamma^D_0(h\mathfrak{p}^{k+1})$ by definition of $v$.
\end{proof}

\begin{lem}
   Assume $\gcd(\mathfrak{p},\eta^*) =1$ and let $h \subset \OD$ be an arbitrary ideal
   with $\gcd(\N(h),\mathfrak{p})=1$. Then 
   $$\left[ (\SL^1(L_d) \cap \Gamma^D_0(h)) :  (\SL^1(L_d) \cap \Gamma^D_0(h\mathfrak{p})) \right] = p + 1 $$	
   holds.
\end{lem}

\begin{proof}
We aim to find a $k \in \NN$ such that the matrices
\begin{eqnarray*}
(I) & Z^{Hi}, & i=1,\ldots,p \\
(II) & Z^{Hk}T &
\end{eqnarray*}
lie in $\Gamma^D_0(h)$ and are pairwise incongruent modulo $\Gamma^D_0(h\mathfrak{p})$. Indeed, we choose $k$ as follows: 
let $k \in \left\{ 1,\ldots,p \right\}$, such that $\mathfrak{p}|kH\eta^-\eta^++1$. This is always possible since 
$\mathfrak{p} \nmid \eta^*$ and $\mathfrak{p} \nmid H$ and $1,..,p$ are incongruent modulo $\mathfrak{p}$. 
Furthermore we know that $k \neq p$ because otherwise it would follow that $\mathfrak{p}|1$.
By definition, it is clear that all the matrices $(I)$ and $(II)$ lie in $\Gamma^D_0(h)$ and that the matrices in $(I)$ are 
pairwise incongruent. Lastly, we calculate $$(Z^{Hk}TZ^{-Hi})_{2,1} = H\eta^-(-(kH\eta^-\eta^++1)i+k).$$
Now suppose $\mathfrak{p}|H\eta^-(-(kH\eta^-\eta^++1)i+k)$. In other words this means 
$\mathfrak{p}|(-i(kH\eta^-\eta^++1)+k)$ which is yet equivalent to $\mathfrak{p}|k$ as $\mathfrak{p}|(kH\eta^-\eta^++1)$. 
This is a contradiction.

\end{proof}

\begin{lem}
 Assume $\gcd(p,\eta^*) =1$ and let $h \subset \OD$ be an arbitrary
 ideal with $\gcd(\N(h),\mathfrak{p})=1$. Then 
 $$\left[ (\SL^1(L_d) \cap \Gamma^D_0(h\mathfrak{p}^\sigma)) :  (\SL^1(L_d) \cap \Gamma^D_0(h\mathfrak{p}^\sigma\mathfrak{p})) \right] = p + 1 $$	
 holds.
\end{lem}

\begin{proof}
Let us construct a matrix $W$ which is a word in $T$ and $Z$ such that the matrices
\begin{eqnarray*}
(I) & WT^{i}, & i=1,\ldots,p \\
(II) & \Id &
\end{eqnarray*}
 lie in $\Gamma^D_0(h\mathfrak{p}^\sigma)$ and are pairwise incongruent modulo $\Gamma^D_0(h\mathfrak{p}^\sigma\mathfrak{p})$. This time, we choose $k$ as follows: let $k \in \left\{ 1,\ldots,p-1 \right\}$, such that $\mathfrak{p}^{\sigma}|k\eta^-\eta^++1$. This is possible since $\mathfrak{p}^{\sigma} \nmid \eta^*$.
 Now suppose that $\mathfrak{p}|k\eta^-\eta^++1$. Since $\gcd(\mathfrak{p},\mathfrak{p}^{\sigma})=1$ we would have 
$\mathfrak{p}\mathfrak{p}^\sigma|k\eta^-\eta^++1$ which would imply $p|k$ (by Remark~\ref{rem:etaproduct} and Lemma~\ref{lem:arithmetic_properties}). 
This is a contradiction. Hence $\mathfrak{p}\nmid k\eta^-\eta^++1$. \\
We now choose $W:=ZT^{k}Z^HT^{-k}Z^{-1}$. Then we have
$$(WT^i)_{2,1} = H \eta^- (k\eta^+\eta^- +1)^2$$
and so all the matrices lie in $\Gamma^D_0(h\mathfrak{p}^\sigma)$ but none of them is equivalent to the identity.
Finally
$$(WT^{x}W^{-1})_{2,1} = -H\eta^{-2}\eta^+(k\eta^-\eta^++1)^4\cdot x$$
holds. As $\mathfrak{p} \nmid H$, $\mathfrak{p} \nmid \eta^*$ and $\mathfrak{p} \nmid (k\eta^-\eta^++1)$ the matrices in $(I)$ are
pairwise incongruent modulo $\Gamma^D_0(h\mathfrak{p}^\sigma\mathfrak{p})$.
\end{proof}

Summarizing we have proven:

\begin{prop} \label{prop_splitting_primes} Assume that $\gcd(p,\eta^*) =1$. Then 
 for all ideals $h \subset \OD$
 with $\gcd(h,\mathfrak{p})=1$ 
$$\left[ (\SL^1(L_d) \cap \Gamma^D_0(h)) :  (\SL^1(L_d) \cap \Gamma^D_0(h\mathfrak{p})) \right] = p+1$$	
holds.
\end{prop}

\paragraph{Divisors of $\eta^*$.} So far we have only treated ideals $\mathfrak{a}$ with $\gcd(\mathfrak{a},\eta^*)=1$.
Before we come to the case $\gcd(\mathfrak{a},\eta^*) \neq 1$ let us analyze the prime divisors
of $m-w$ and $n-w$. There are two different types of prime ideals
of norm $p$, namely
$$\mathfrak{p}_1 = [p,w] \ \ \textrm{and} \ \ \mathfrak{p}_2 = [p, d+w].$$
If the prime ideal is of type $1$ and divides $m-w$ it follows that $p|m$ and that $\mathfrak{p}_1 \nmid d - w$
since $p \nmid d$. If the prime ideal is of type~$2$ and divides $m-w$, it follows that $p|m+d$ and that $\mathfrak{p}_2 | d- w$.
Now we choose $m=d-1$ and $n=2$ whenever this is possible by Corollary~\ref{cor:classification} and $m=d-2$ and $n=3$ 
otherwise. Note that the only possible common divisor of $m-w$ and $n-w$ are in each case the prime ideal divisors of~$2$. 

\begin{rem}
  By the choice of $m$ there does not exist a $p \in \ZZ$ with $p|d$ and $p|m$. Thus, the main theorem is proven for all
  odd divisors of the discriminant.
\end{rem}


\begin{lem}
   Assume $\mathfrak{p}|\eta^-$ and let $h \subset \OD$ be an arbitrary ideal with $\gcd(h,\mathfrak{p})=1$. Then for all $k \in \mathbb{N}$ 
$$\left[ (\SL^1(L_d) \cap \Gamma^D_0(h\mathfrak{p}^k)) :  (\SL^1(L_d) \cap \Gamma^D_0(h\mathfrak{p}^{k+1})) \right] = p $$	
holds.
\end{lem}

\begin{proof} \textit{Case $n=2$:} Note that $\mathfrak{p} \nmid d$ since $p$ splits. Since $\mathfrak{p}$ is a prime ideal it is either of the form $[p,w]$ or $[p,d+w]$. As $\mathfrak{p}|\eta^-$ it must even be of the form $[p,w]$. Hence we have that $\mathfrak{p} \nmid (d-w)$ but $\mathfrak{p}^\sigma | (d-w)$. We set $H'=\N(h(h,\mathfrak{p}^\sigma)^{-1})$ and claim that the matrices $E^{H'p^ki}, 1 \leq i \leq p$ 
are a set of coset representatives. Now assume that $\mathfrak{p}^{k+1} | H'(d-w)p^k i$. This yields $\mathfrak{p}| (d-w)i$ (Lemma~\ref{lem:arithmetic_properties}). Since $\mathfrak{p} \nmid d-w$ the claim follows.\\
\textit{Case $n=3$:} We have to treat the case $\mathfrak{p}|\eta^-$ and $\mathfrak{p}|2$ separately. 
If $\mathfrak{p}=[2,d+w]$ divides $3-w$ then $d-3$ is odd and thus $\mathfrak{p}$ does not divide $F_{2,1}$. Hence the claim 
follows.
\end{proof}

\begin{lem}
  Assume $\mathfrak{p}|\eta^-$ and $\mathfrak{p} \nmid 2$ and let  $h \subset \OD$ be an arbitrary ideal with $\gcd(h,\mathfrak{p})=1$. Then 
$$\left[ (\SL^1(L_d) \cap \Gamma^D_0(h)) :  (\SL^1(L_d) \cap \Gamma^D_0(h\mathfrak{p})) \right] = p+1 $$	
holds.
\end{lem}

\begin{proof}
  \textit{Case (1) $\mathfrak{p} \nmid d-n$}: Note that since $\mathfrak{p}| \eta^-$ and $\mathfrak{p} \nmid d-n$ we get that $\mathfrak{p} \nmid (d-w)$ but $\mathfrak{p}^\sigma | (d-w)$. Let
  $H'=\N(h(h,\mathfrak{p}^\sigma)^{-1})$. We claim that the matrices
\begin{eqnarray*}
(I) & E^{H'}T^k & k=1,\ldots,p \\
(II) & \Id &
\end{eqnarray*}
lie in $\Gamma^D_0(h)$ and are pairwise incongruent modulo $\Gamma^D_0(h\mathfrak{p})$. Since 
$$(E^{H'}T^k)_{2,1} = H'(d-w)$$
none of the matrices in $(I)$ is congruent to the identity. On the other hand
$$(E^MT^kE^{-M})_{2,1} = -H'^2\eta^+(d-w)^2 \cdot k$$
and thus the matrices $(I)$ are pairwise incongruent modulo $\Gamma^D_0(h\mathfrak{p})$.\\
  \textit{Case (2) $\mathfrak{p} | d-n$}: By case (1) we may assume that $\mathfrak{p} \nmid \N(h)$. The arguments work as above after replacing $E$ by $F$.
\end{proof}

For the divisors of $\eta^+$, Nori's theorem significantly facilitates the proof.

\begin{lem}
   Assume $\mathfrak{p}|\eta^+$ and let $h \subset \OD$ be an arbitrary ideal with $\gcd(h,\mathfrak{p})=1$. Then 
$$\left[ (\SL^1(L_d) \cap \Gamma^D_0(h) :  (\SL^1(L_d) \cap \Gamma^D_0(h\mathfrak{p}) \right] = p+1 $$	
holds.
\end{lem}

\begin{proof}
  \textit{Case $n=2$:} Note that since $\mathfrak{p}| \eta^+=2-w$ we have that $\mathfrak{p} \nmid 3-w$. By what we have proven so far, we may
  assume that $\gcd(\N(h),\mathfrak{p})=1$. The matrices $F^{Mi}$ and $T^{Mi}$ all lie in $\Gamma^D_0(h)$ and they
  all yield different elements when they are projected to $\SL_2(\OD/\mathfrak{p}\OD) \cong \SL_2(\mathbb{F}_p)$. By
  the corollary to Nori's theorem (Corollary~\ref{cor:Nori}), this map has to be surjective and therefore $\SL^1(L_d) \cap \Gamma_D^0(h\mathfrak{p})$ must have
  the maximal possible index in $\SL^1(L_d) \cap \Gamma^D_0(h)$.\\
  \textit{Case $n=3$:} The same reasoning as above is possible since $\eta^+=3-w$ and $F_{1,2}=(4-w)$. 
\end{proof}

\paragraph{The divisors of $2$.} We finally come to the missing prime ideal divisors of~$2$. 

\begin{lem} If~$2$ splits, let $\mathfrak{p}_2 = (2,\eta^-)$ and let $h \subset \OD$ be an arbitrary ideal
 with $\gcd(h,2)=1$. If $L_d \in \mathcal{A}_d$, then 
 $$\left[ (\SL^1(L_d) \cap \Gamma^D_0(h)) :  (\SL^1(L_d) \cap \Gamma^D_0(2h)) \right] = 6$$
 and	
 $$\left[ (\SL^1(L_d) \cap \Gamma^D_0(h)) :  (\SL^1(L_d) \cap \Gamma^D_0(h\mathfrak{p}_2)) \right] = 2.$$
\end{lem}

\begin{proof}
  It can easily be checked that the matrices $T^H, Z^H, Z^HT^H,E^H, E^HZ^H$ and $E^HZ^HT^H$ are inequivalent modulo $\Gamma^D_0(2h)$.
  By Weitze-Schmithüsen's Theorem respectively Corollary~\ref{cor:schmitt} the index is at most $6$. Thus the claim follows.
\end{proof}

\begin{lem} If $2|d$ then 
 $$\left[ (\SL^1(L_d) \cap \Gamma^D_0(h)) :  (\SL^1(L_d) \cap \Gamma^D_0(2h)) \right] = 4.$$	
\end{lem}

\begin{proof}
  The matrices $T^H,Z^H,Z^HT^H$ are inequivalent modulo $\Gamma^D_0(2h)$ and by Corollary~\ref{cor:schmitt} the index is either~$2$ or~$4$.
\end{proof}
This finishes the proof of Theorem~\ref{thm:main_thm}, $(1)$ and $(3)$.\\[12pt]
The most special case is the case where $d$ is odd and $L_d \in \mathcal{B}_d$. Then $m=3$ and $\eta^-=3-w$. Since the norm of
$\eta^-$ is positive we have that $\mathfrak{p}_2^\sigma | 3-w$ (but $\mathfrak{p}_2 \nmid 3-w$). So we are now only
interested in $\Gamma^D_0(\mathfrak{p}_2^\sigma)$.
\begin{rem} By considering the matrix $F$ we immediately see that for all ideals $h \subset \OD$ 
 with $\gcd(h,\mathfrak{p}_2)=1$
 $$\left[ (\SL^1(L_d) \cap \Gamma^D_0(h) :  (\SL^1(L_d) \cap \Gamma^D_0(h\mathfrak{p}_2^\sigma) \right] \geq 2$$
 holds if $L_d \in \mathcal{B}_d$.
\end{rem}
So it remains to show that the index is not greater than~$2$. To prove this, it is convenient to consider the subgroup
$\SL^1(L_d) \cap \Gamma^D(2)$ instead and to show that this index is at most $\frac{1}{3} [\SL_2(\OD): \Gamma^D(2)]$. By projection on the second factor of
the homology Veech group and by the arguments given in the proof of \cite[Theorem~3~(ii)]{WS12} it suffices to show that one
of the non-integral Weierstra\ss \ points on $E_2$ can be distinguished from the others. This is clear by Theorem~\ref{thm:weierstrass}
since $E_2$ has only one integral Weierstra\ss \ point. Thus we have proven:
\begin{prop} We have
 $$\left[ \SL^1(L_d) :  (\SL^1(L_d) \cap \Gamma^D(2)) \right] \leq \frac{1}{3}  [\SL_2(\OD): \Gamma^D(2)] $$
 if $L_d \in \mathcal{B}_d$.
\end{prop}

\begin{cor}
  For all ideals $h \subset \OD$  with $\gcd(h,\mathfrak{p}_2^\sigma)=1$
 $$\left[ (\SL^1(L_d) \cap \Gamma^D_0(h) :  (\SL^1(L_d) \cap \Gamma^D_0(h\mathfrak{p}_2^\sigma) \right] = 2$$
 holds if $L_d \in \mathcal{B}_d$.
\end{cor}

This finishes the proof of Theorem~\ref{thm:main_thm}, $(2)$.

\appendix
\section{Appendix} \label{App}

In the appendix we collect some results on  quadratic orders $\OD$ of square discriminant. In particular, we will give proofs of the results mentioned in the main part of the paper.\\[12pt] Let $K =\QQ \oplus \QQ$. A quadratic order of square discriminant is a subring $\mathcal{O}$ of $K$ that is also a finitely generated $\ZZ$-module such that $1 \in \mathcal{O}$ and $\mathcal{O} \otimes \QQ = K$. Recall that any such quadratic order is of the form 
$$\OD = \left\{ (x,y) \in \ZZ \times \ZZ \ | \ x \equiv y \mod d \right\}$$
with $D=d^2$.
\paragraph{Norm and trace.} Analogously as in the non-square case we have:
\begin{rem} An element $z \in K$ is in $\QQ$ if and only if $z = z^\sigma$, since $\QQ$ is diagonally embedded into $K$.\end{rem}
Recall that if $D \in \mathbb{N}$ is square-free, then the order $\OD$ is the ring of integers of $K=\QQ(\sqrt{D})$ and an element
$z \in K$ is in $\OD$ if and only both trace and norm of $z$ lie in $\ZZ$.\footnote{This in not true any more if the discriminant $D$ is not square-free as the example $D=45$ and $z=(2+5w)/3$ shows.} In accordance with this, a similar property also holds if $D=d^2$ is the square of a prime number. We can however not expect that $z \in \OD$ if and only if $\tr(z)$ and $\N(z)$ are in $\ZZ$, as we can see by the example $d=5$ and $z=(1,2)$. Indeed, one has to additionally impose some congruence conditions on the trace and the norm.

\begin{lem} Let $z \in K$. If $D=d^2$ and $d$ has no prime divisor of order greater than $1$, then  $z \in \OD$ if and only if $\tr(z), \N(z) \in \ZZ$ and
$\tr(z) \equiv 2v \mod d$ and $\N(z) \equiv v^2 \mod d$ for some $v \in \ZZ$.
\end{lem}

\begin{proof}
  If $z \in \OD$ then $z=(x,y)$ for some $x,y \in \ZZ$. Thus $\tr(z)=(x+y,x+y)$ and $\N(z)=(xy,xy)$ and 
hence $\tr(z), \N(z) \in \ZZ$. Moreover $x \equiv y \mod d$ implies that $x+y \equiv 2x \mod d$ and $xy \equiv x^2 \mod d$. \\ 
On the other hand, let $z=(x,y)$ for some $x,y \in \QQ$ and let $\tr(z)$ and $\N(z)$ be in $\ZZ$. Since $x+y$ and 
$xy$ are both in $\ZZ$, also $x$ and $y$ are in $\ZZ$. Then $\tr(z) \equiv 2v \mod d$ implies
that $y \equiv 2v - x \mod d$. Inserting this into $\N(z) \equiv v^2 \mod d$ gives $(x-v)^2 \equiv 0 \mod d.$
Since $d$ has no quadratic term we have $x \equiv v \mod d$ and thus $x \equiv y \mod d.$ 
\end{proof}
We will from now on until the end of the appendix exclusively restrict to the case where $D=d^2$. The following lemma describes some basic arithmetic properties in this situation.
\begin{lem} \label{lem:arithmetic_properties} \begin{itemize}
		\item[(i)]  If $b \in \OD$ is not a zero divisor and $a,c \in \OD$ then $ab|cb$ if and only if $a|c$. 
		\item[(ii)] If $p \in \ZZ$ is a prime number and $a,b \in \ZZ$ then $p|a+wb$ if and only if $p|a$ and $p|b$.
		\item[(ii)] If $p,m,x \in \OD$ and $p \in \ZZ$ is a prime number with $p|d$ and $p|mx$. Then $p \nmid \N(m)$ implies $p|x$.
	\end{itemize}
\end{lem}

\begin{proof}  $(i)$ and $(ii)$ are clear by definition.\\
	$(iii)$ Let $m=(m_1,m_2)$ and $x=(x_1,x_2)$. The relation $p|mx$ implies $\N(p) | \N(mx)$ or in other words $p^2| m_1m_2x_1x_2$. Since
	$p \nmid \N(m)$ hence $p|x_1x_2$ and since $p|d$ we must have $p|x_1$ and $p|x_2$, i.e. $x_1=pk_1$ and $x_2=x_1+jd =pk_1+pk_2$
	with $k_1,k_2 \in \ZZ$. Then $(p,p)|(pk_1m_1,(pk_1+pk_2)(m_1+ld))$ is equivalent to
	$$(p,p)|(pm_1k_1,p(m_1k_1+m_1k_2+k_1ld+k_2ld))$$
	which implies $d|m_1k_2+(k_1+k_2)ld$. From this it follows either that $p|m_1$ which contradicts the fact $p \nmid \N(m)$ or
	$p|k_2$. In the latter case $x_2=p(k_1+d\widetilde{k_2})$ and so $p|x$ as we claim.
\end{proof}

As in the case of non-square-discriminants, $\OD$ is a Noetherian ring: this is true, because $\mathds{1}:=(1,1)$ and $w:=(0,d)$ form a $\ZZ$-basis of the $\ZZ$-module $\OD$. We call this basis the \textbf{standard basis} of $\OD$.


\paragraph{Ideals.} So far, the notion of the norm has only been defined for elements in $K$ but not for ideals. Recall that a \textbf{regular} ideal refers to an ideal containing a non-zero divisor.
\begin{lem} A prime ideal $\mathfrak{p}$ of $\OD$ is maximal if and only if it is regular. \end{lem}
 
 \begin{proof} If $\mathfrak{p}$ is a maximal $\OD$-ideal, then $\mathfrak{p} \cap \ZZ$ is maximal and so $\mathfrak{p} \cap \ZZ \neq 0$, i.e. $\mathfrak{p}$ is regular. Conversely, if $x \in \mathfrak{p}$ is regular, then $0 \neq \N(x) \in \mathfrak{p} \cap \ZZ$. Thus, $\mathfrak{p} \cap \ZZ$ is maximal and hence also $\mathfrak{p}$ is. 
 \end{proof}
 The definition of the norm of a reguler ideal perfectly generalizes the norm of an element. 
 
\begin{lem} \label{lem:app:norm}
 For $z \in \OD$ we have $\N((z))=|\N(z)|.$ 
\end{lem}

In particular, it follows that $\N((z))=z^2$ for all $z \in \ZZ$ and $\N_D((z))=0$, if $z \in \OD$ is a zero divisor. 

\begin{proof}
  The linear map $T_z: \OD \to \OD$ given by $T_z(\alpha) = z\alpha$ has determinant $\N(z)$. Thus, $\N((z))= [\OD : z\OD] = |\det(T_z)| = |	\N(z)|$.
\end{proof}

\paragraph{Ideals as modules.} Let us now describe all ideals in $\OD$. It is well-known that every ideal of $\OD$ is also a $\ZZ$-module. 
The point of view that ideals are modules is very useful for giving a list of prime ideals. Moreover it allows us to calculate $\textrm{Spec} \ \OD$. As in the case of non-square discriminants, it is essential to see that every $\ZZ$-module in $\OD$ is generated by at most two elements. 

\begin{prop} \label{prop:modules_appendix} Let $M \subset \OD$ be a $\ZZ$-module in $\OD$. Then there exist integers $m,n \in \ZZ_{\geq 0}$ and $a \in \ZZ$
 such that $$M = [n\mathds{1}; a\mathds{1}+mw]:= n\mathds{1}\ZZ \oplus (a\mathds{1}+mw)\ZZ.$$
\end{prop}

\begin{proof} Consider the subgroup $H:= \left\{s \in \ZZ: r\mathds{1}+sw \in M \right\}$ of $\ZZ$. As $H$ is a subgroup
of $\ZZ$, it is of the form $m\ZZ$ for some $m \geq 0$. By construction, there exists an $a \in \ZZ$ with
$a\mathds{1} + mw \in M$. Furthermore we know that $M \cap \ZZ\mathds{1}$ can be regarded a subgroup of $\ZZ$ and so
$M \cap \mathds{1}\ZZ = n \mathds{1} \ZZ$ for some $n \geq 0.$  We claim that $M = n\mathds{1}\ZZ \oplus (a\mathds{1} + mw) \ZZ.$
The inclusion $\supseteq$ is evident. Hence let us assume that $r\mathds{1} + sw \in M$. Since $s \in H$ we have
$s = um$ for some $u \in \ZZ$, and thus
$$r\mathds{1} -ua\mathds{1} = r\mathds{1} +sw - u(a \mathds{1} +mw) \in M \cap \mathds{1}\ZZ.$$
Hence $r-ua = nv$. But then
$$r\mathds{1} + sw = (r-ua)\mathds{1} + u(a\mathds{1}+mw) = nv\mathds{1} + u(a\mathds{1}+mw) \in n\mathds{1}\ZZ \oplus (a\mathds{1} + mw) \ZZ.$$
\end{proof}

As it simplifies notation and cannot cause any confusion the symbol $\mathds{1}$ is usually omitted when embedding $\ZZ$ into $\OD$. In other words, we write every $\ZZ$-module in $\OD$ as $[n;a+mw]$ for some $a,n,m \in \ZZ.$\\[12pt]
Since every ideal of $\OD$ is also $\ZZ$-module, it is generated by at most two elements. The converse is not true since
e.g. $M=[1;0] = \ZZ$ is a $\ZZ$-submodule of $\OD$, but not an ideal. We therefore now describe under which conditions
on $a,m,n$ the $\ZZ$-module $M$ is also an ideal. These conditions are just the same as in the case of non-square
discriminants.

\begin{prop} \label{prop:ideal_conditions_appendix}
 A regular $\ZZ$-module $M=[n;a+mw]$ is an ideal if and only if $m|n$, $m|a$, i.e. $a=mb$ for some $b \in \ZZ$, and
 $n|m\N(b+w)$.
\end{prop}

\begin{proof}
 Suppose that $M$ is an ideal and consider the group $H$ from the proof of Proposition~\ref{prop:modules_appendix}. 
 Then $c \in M \cap \ZZ$ implies $cw \in M$ and hence $c \in H$. This shows that $n\ZZ = M \cap \ZZ \subset H = m\ZZ$ 
 or in other words that $m|n$. Observe that $w^2 = dw$. Since $M$ is an ideal, $a+mw \in \M$ implies that
 $(a+mw)w=(a+md)w \in M$. By definition of $H$ we therefore have that $a \in H$ and hence $m|a$. Finally, we set
 $\beta:=a+mw=m(b+w)$. Then $\beta \in M$ yields $\beta(b+w^\sigma) \in M$. Hence $n| m\N(b+w).$\\
 Now suppose that all the divisibility relations are fulfilled by $M$. It suffices to check that $nw$ and
 $(a+mw)w$ both lie in $M$. We have
 $$nw=\frac{n}{m}mw = \frac{n}{m} (a+mw) - \frac{n}{m}a = \frac{n}{m}(a+mw)-bn$$
 and so $nw \in M$ since $m|n$. And
 \begin{align*}
 (a+mw)w & = aw+mw^2 = m(b+d)w\\
 & = (b+d)(a+mw) - mb(b+d)\\ 
 & = (b+d)(a+mw)-m\N(b+w)
 \end{align*}
 implies $(a+mw)w \in M$, because $n|m\N(b+w).$
\end{proof}

For an arbitrary ideal $\mathfrak{a} = [n;a+mw]$, it is straightforward to check (by giving an explicit list of representatives) that $\N(\mathfrak{a})=|mn|$. 
Note that if $M=[n;a+mw]$ is an ideal then its conjugated module is given by $[n;a+mw]^\sigma=[n;a+md-mw].$ If $\mathfrak{a}=((x,y))$ is a principal ideal then also $\mathfrak{a}^\sigma=((y,x))$ is a principal ideal. 

\begin{cor} \label{cor:prime_conditions:appendix}
  Every ideal of prime norm $p$ is of the form $[p;a+w]$ for some $a \in \ZZ$ with $p|\N(a+w).$ These ideals
  are indeed prime ideals.
\end{cor}

\begin{proof}
  The first assertion is clear from Proposition~\ref{prop:ideal_conditions_appendix}. The second assertion follows from the fact that if $[p;a+w]$ is an $\OD$-ideal, then it has index $p$ in $\OD$ and so it is a maximal ideal.
\end{proof}
The corollary shows that there does not exist any inert prime number if $D$ is a square because it is always possible to find an $a \in \ZZ$ such that $p|\N(a+w)$, e.g. $a=0$. Furthermore, it puts us into the position to count the number of different  prime ideals of  norm $p$ if $p$ is a prime number. This is paves the way towards a ramification theory of prime numbers over $\OD$.

\begin{thm} \label{prop:Spec}
  Let $p \in \ZZ$ be a prime number. If $p|D$ then there exists exactly one prime ideal $\mathfrak{a}$ of norm $p$
  and $\mathfrak{a}^\sigma=\mathfrak{a}$. Otherwise there exist exactly two different prime ideals $\mathfrak{a}, \mathfrak{b}$
  of norm $p$ and $\mathfrak{a}^\sigma=\mathfrak{b}$.
\end{thm}

\begin{proof}
  Let $p \in \ZZ$ be a prime number with $p|D$. Then every ideal of norm $p$ is of the form
  $\mathfrak{a}=[p;a+w]$ with $p|a(a+d)$. As $p \in \mathfrak{a}$ we may without loss of generality assume that 
  $0 \leq a \leq p-1$.  Since $p|d$ we get $p|a$ and therefore $a=0$. So there exists exactly one ideal
  of norm $p$ if $p|d$. It is then clear that $\mathfrak{a}^\sigma=\mathfrak{a}$.\\
  If $p \in \ZZ$ is a prime number with $p\nmid d$ then we may again assume that $0 \leq a \leq p-1$. So there remain
  the two possibilities $p|a$ and $p|(a+d)$. These ideals $\mathfrak{a}$ and $\mathfrak{b}$ are indeed different since $p \nmid d$ and $\mathfrak{a}^\sigma=\mathfrak{b}$.
\end{proof}

\paragraph{Ramification.} From this the ramification theory for prime numbers over $\OD$ can be deduced. In order to do this, we have to analyze the multiplication of two prime ideals. Let us first assume that $p$ is a prime number with $p|d$ and let $\mathfrak{a}=[p,w]$ be the unique  prime ideal of norm $p$. Then $\mathfrak{a}^2=[p,w][p,w] = (p)[p, w, w, d/pw]$. Hence $\mathfrak{a}^2 = (p)[p,w] \neq (p)$. In particular, $(p)$ cannot be further decomposed and the norm is not multiplicative. 
 If $p \nmid d$ let $\mathfrak{a}=[p,w]$ and $\mathfrak{b}=[p,d+w]$ be the two different ideals of norm $p$. Hence
 \begin{align*}
 [p,w][p,d+w]=[p^2,pw,pd+pw,dw+w^2]=p[p,w,d] = (p).
  \end{align*}
 Recall that an ideal is called \textbf{irreducible} if it cannot be written as the intersection of two larger ideals. Hence we have established:

\begin{thm} \label{thm:ramification_of_primes_appendix}
 Let $p \in \ZZ$ be a prime number. 
 \begin{itemize}
  \item[(i)] If $p \nmid d$ then $(p)=\mathfrak{p}\mathfrak{p}^\sigma$ for a prime ideal $\mathfrak{p}$ of norm 
	$p$ with $\mathfrak{p} \neq \mathfrak{p}^\sigma$, i.e. $p$ splits.
  \item[(ii)] If $p|d$ then $(p)$ is an irreducible ideal which is not prime. 
 \end{itemize} 
\end{thm}

\begin{cor} \label{cor:prime_decomposition_appendix}
	Every ideal $\mathfrak{a} \subset \OD$ with $\gcd(\N(\mathfrak{a}),d)$ can be uniquely written as product of prime ideals.
\end{cor}
\begin{proof} We have $\mathfrak{a}\mathfrak{a}^\sigma = \N(\mathfrak{a})$. From Theorem \ref{thm:ramification_of_primes_appendix} it follows that $\N(\mathfrak{a})$ can be uniquely written as product of prime ideals. Thus, this is also true for $\mathfrak{a}$.	
\end{proof}

\paragraph{The special linear group.} In this paragraph, we only prove the claim of Proposition~\ref{prop_exact_sequence} with the help of the following two lemmas.

\begin{lem} Let $R,S$ be two commutative rings such that there exists a surjective homomorphism of rings $f: S \to R$.
If $\SL_2(R)$ is generated by elementary matrices the induced map $\SL_2(S) \to \SL_2(S)$ is also surjective. \end{lem}

\begin{proof} Any elementary matrix over $R$ lifts to an elementary matrix of $S$. \end{proof}

\begin{lem} If $R$ is a finite commutative ring, then $\SL_2(R)$ is generated by elementary matrices. \end{lem}

\begin{proof}
  Every finite commutative ring is a direct product of local rings. Since the claim is true for local rings (see e.g. \cite[Chapter~2.2]{Ros94})
  this finishes the proof.
\end{proof}

\begin{proof}[Proof (of Proposition~\ref{prop_exact_sequence})] By definition $\Gamma^D(\mathfrak{a})$ is the kernel of the
projection $\SL_2(\OD) \to \SL_2(\OD/\mathfrak{a})$. The projection map is surjective by the preceding two lemmas.
\end{proof}

\textsc{Hochschule Ruhr West, Duisburger Str. 100, D-45479 M\"ulheim an der Ruhr}\\
\textit{E-mail address:} \texttt{christian.weiss@hs-ruhrwest.de}


\begin{thebibliography}{xxx}
 \bibitem[AM69]{AM69} Atiyah M, MacDonald I., ``Introduction to Commutativ Algebra'', Addison-Wesley, Reading, 1969.
 \bibitem[Bai07]{Bai07} Bainbridge, M. ``Euler Characteristics of Teichmüller Curves in genus two'', Geometry \& Topology 11, 1887-2013 (2007).
 \bibitem[BL04]{BL04} Birkenhake, C., Lange, H., ``Complex Abelian Varieties'', Springer, Berlin Heidelberg New York (2004).
 \bibitem[ER12]{ER12} Ellenberg, J., McReynbolds D., ``Arithmetic sublattices of $\SL(2,\ZZ)$'', Duke Math. J., 161(3), 415-429 (2012).
 \bibitem[FK92]{FK92} Farkas, H., Kra, I., ``Riemann Surfaces'', Springer, Berlin Heidelberg New York (1992).
 \bibitem[HL06]{HL06} Hubert, P. and Leli\`{e}vre, S., ``Prime arithmetic Teichmüller discs in $\mathcal{H}(2)$'', Israel 
 Journal of Mathematics 151, 501-526 (2006).
 \bibitem[HS06]{HS06} Hubert, P. and Schmidt, T. ``An introduction to Veech surfaces'', in: \textit{Handbook of Dynamical Systems}, vol. 1B, ed. by A. Katok, B. Hasselblatt, Elsevier B.V., Amsterdam (2006).
 Journal of Mathematics 151, 501-526 (2006).
 \bibitem[Kan03]{Kan03} Kani, E., ``Hurwitz spaces of genus 2 covers of an elliptic curve'', Collect. Math., 54(1), 1-51 (2003).
 \bibitem[Kap11]{Kap11} Kappes, A., ``Monodromy Representations and Lyapunov Exponents of Origamis'', \textit{PhD-thesis}, Karlsruhe (2011).
 \bibitem[Kil08]{Kil08} Kilford, L.J.P., ``Modular forms - A classical and computational introduction'', Imperial College Press, London (2008).
 \bibitem[Kuh88]{Kuh88} Kuhn, R., ``Curves of genus 2 with split Jacobians'', Trans. of the AMS, 307(1), 41-49 (1988).
 \bibitem[McM03]{McM03} McMullen, C., ``Billiards and Hilbert modular surfaces'', J. Amer. Math. Soc., 16 (4), 857-885 (2003).
 \bibitem[McM05]{McM05} McMullen, C., ``Teichmüller curves in Genus Two: Discriminant and Spin'', Math Ann., 333, 87-130 (2005).
 \bibitem[Möl05]{Möl05} Möller, M., ``Teichmüller curves, Galois actions and GT-relations'', Math. Nachrichten 278 No. 9 (2005).
 \bibitem[Möl11]{Möl11} Möller, M., ``Teichmüller Curves from the Viewpoint of Algebraic Geometry'', \textit{preprint} (2011).
 \bibitem[MZ16]{MZ16} Möller, M., Zagier D.B., ``Modular embeddings of Teichmüller curves'', Composito Math, 152, 2263-2349 (2016).
 \bibitem[Muk11]{Muk11} Mukamel, R., ``Orbifold points on Teichmüller curves and Jacobians with complex multiplication'', \textit{PhD-thesis}, Harvard (2011).
 \bibitem[Muk13]{Mur13} Mukamel, R., ``Fundamental domains and generators for lattice Veech groups'', \textit{preprint} (2013).
 \bibitem[Nor87]{Nor87} Nori, M., ``On subgroups of $\GL_n(F_p)$'', Invent. math. 88, 257-275 (1987).
 \bibitem[Rap12]{Rap12} Rapinchuk, A, ``On strong approximation for algebraic groups'', \textit{preprint}, arXiv:1207.4425 (2012).
 \bibitem[Ros94]{Ros94} Rosenberg, J., ``Algebraic K-Theory and its applications'', Springer, Berlin Heidelberg New York (1994).
 \bibitem[Vee89]{Vee89} Veech, W.A., ``Teichmüller curves in moduli space, Eisenstein series and an application to triangular billiards'', Invent. Math., 97 (3), 553-583 (1989).
 \bibitem[Wei08]{Wei08} Wei\ss, C., ``Hecke Operators and Orthogonality on $\Gamma_1[N]$'', \textit{diploma thesis}, Heidelberg (2008).
 \bibitem[Wei12]{Wei12} Wei\ss, C., ``Twisted  Teichm\"uller  curves'',  Lecture  Notes  in  Mathematics,  vol.  2041  Springer Berlin (2014).
 \bibitem[WS05]{WS05} Weitze-Schmithüsen, G., ``Veech groups of Origamis'', \textit{PhD-thesis}, Karslruhe (2005).
 \bibitem[WS12]{WS12} Weitze-Schmithüsen, G., ``The deficiency of being a congruence group for Veech groups of Origamis'', arXiv:1208.1936 (2012).

\end{thebibliography}
\end{document}